\documentclass[a4paper, web]{ieeecolor}
\usepackage{generic}
\usepackage{lipsum}
\usepackage{cite}
\usepackage{amsmath,amssymb,amsfonts}
\usepackage{algorithm,algorithmic}
\usepackage{hyperref}
\hypersetup{hidelinks=true}
\usepackage{textcomp}

\makeatother

\usepackage{graphicx,subfigure}
\usepackage{amsthm}
\usepackage{tabu}
\usepackage{breqn}
\usepackage{verbatim}
\usepackage{enumerate}

\usepackage{esvect}
\usepackage{footnote}
\usepackage{siunitx}
\usepackage{color,xcolor}
\usepackage{url}
\usepackage[makeroom]{cancel}
\usepackage{todonotes}

\newtheorem{theorem}{Theorem}[section]
\newtheorem{corollary}[theorem]{Corollary}

\newtheorem{lemma}[theorem]{Lemma}
\newtheorem{proposition}[theorem]{Proposition}
\newtheorem{remark}{Remark}
\newtheorem*{problem statement}{Problem Statement}

\newtheorem{assumption}{Assumption}

\newcommand{\real}{\mathbb{R}}

\newcommand{\lie}{\mathcal{L}}

\newcommand{\Cc}{{\mathcal{C}}}
\newcommand{\Nc}{{\mathcal{N}}}
\newcommand{\Fc}{{\mathcal{F}}}
\newcommand{\Kc}{{\mathcal{K}}}
\newcommand{\Sc}{\mathcal{S}}

\DeclareMathOperator*{\argmax}{arg\,max}
\DeclareMathOperator*{\argmin}{arg\,min}

\newcommand{\dotprod}[2]{\ensuremath{\left\langle #1, #2 \right\rangle}}

\newcommand{\longthmtitle}[1]{\mbox{}\emph{(#1):}}
\newcommand{\setdef}[2]{\{#1 : #2\}}

\newcommand{\norm}[1]{\left\lVert#1\right\rVert}

\allowdisplaybreaks
\renewcommand{\baselinestretch}{.98}

\begin{document}
\title{Safe Feedback Optimization\\through Control Barrier Functions}
\author{Giannis Delimpaltadakis$^*$ \quad Pol Mestres$^*$ \quad Jorge Cort\'es \quad W.P.M.H. Heemels
\thanks{$^*$Equal contribution.}
\thanks{Giannis Delimpaltadakis is with the Robust and Intelligent Autonomous Systems lab, AI4I institute, Torino, Italy. Maurice Heemels is with the Control Systems Technology (CST) section, Mechanical Engineering, Eindhoven University of Technology. Pol Mestres is with the Department of Mechanical and Civil Engineering, California Institute of Technology, USA. Jorge Cortés is with the Department of Mechanical and Aerospace
Engineering, University of California San Diego, USA. Most of the work was conducted while G. Delimpaltadakis was with CST, Eindhoven University of Technology, and P. Mestres was with the University of California, San Diego, USA. Emails: \texttt{ioannis.delimpaltadakis@ai4i.it, mestres@caltech.edu, cortes@ucsd.edu, m.heemels@tue.nl}. \newline \indent This research is partially funded by the European Research Council (ERC) under the Advanced ERC grant PROACTHIS, no. 101055384 and by AFOSR Award FA9550-23-1-0740.}}
\maketitle

\begin{abstract}
 Feedback optimization refers to a class of methods that steer a control system to a steady state that solves an optimization problem. Despite tremendous progress on the topic, an important problem remains open: \emph{enforcing state constraints at all times}. The difficulty in addressing it lies on mediating between the safety enforcement and the closed-loop  stability, and ensuring the equivalence between closed-loop equilibria and the optimization problem's critical points. In this work, we present a feedback-optimization method that enforces state constraints at all times employing high-order control-barrier functions. We provide several results on the proposed controller dynamics, including well-posedness, safety guarantees, equivalence between equilibria and critical points, and local and global (in certain convex cases) asymptotic stability of optima. Various simulations illustrate our results.
\end{abstract}

\section{Introduction}
Feedback optimization refers to a class of methods that regulate a system to a steady state, that solves an optimization problem; see \cite{hauswirth2021optimization} for a survey. These methods design controllers in the form of gradient flows, and place them in feedback loop with the system, while typically enforcing that the controller dynamics are sufficiently slower than the system dynamics (\emph{timescale separation}; a notable exception is \cite{BIANCHI2025101308}), to guarantee closed-loop stability. Compared to the feedforward approach of solving the optimization problem offline and afterward steering the system to the computed optimum, feedback optimization enjoys superior robustness, as it can handle unknown and time-varying dynamics and objective functions.

Research on feedback optimization has surged~\cite{MC-ED-AB:20,jokic2009constrained,LSPL-JWSP-EM:24}, with applications including power systems~\cite{dall2016optimal,colot2024optimal,brunner2012feedback,KH-JPH-KU:14, chen2020distributed}, traffic control \cite{GB-JC-JIP-EDA:22-tcns}, smart buildings \cite{belgioioso2021sampled} and communication networks \cite{wang2011control}. However, although designs with input constraints have been proposed (e.g., \cite{hauswirth2020timescale,YC-LC-JC-EDA:23-csl,haberle2020non}), a fundamental problem remains unsolved: \emph{enforcing state constraints at all times}. The challenging nature of the problem lies in: a) proposing a safety-enforcement mechanism that retains equivalence between the closed-loop's equilibria and the optimization problem's critical points\footnote{Points satisfying the KKT conditions \cite{NA-AE-MP:20}.}, and b) reconciling closed-loop stability with safety enforcement, which requires controller dynamics to operate at the same timescale as the system, with potentially destabilizing effects. This paper provides a solution addressing both issues.

\subsubsection*{Contributions}
In this work, we propose a feedback optimization method, that enforces both input and state constraints at all times. The proposed controller is a \emph{safe gradient flow} (SGF; see \cite{allibhoy2023control}), and employs \emph{high-order control-barrier functions} (CBFs; see \cite{xiao2021high}), towards enforcing state constraints. The controller dynamics are given by a quadratic program (QP), which can be solved online. We accompany the proposed controller with several theoretical results:
\begin{itemize}
    \item We provide conditions for feasibility of the controller's QP and for the existence and uniqueness of solutions of the closed loop, guaranteeing well-posedness;
    \item We establish conditions under which a subset of the state and input constraint set is forward invariant, thus guaranteeing safety. This subset 
    can be made arbitrarily close to the constraint set by parameter tuning, under compactness assumptions.
    It also contains all feasible points of the feedback-optimization problem (including critical points and optima). Finally, the constraints may be of arbitrary relative degree;
    \item We prove the equivalence between equilibria and critical points under different sufficient conditions. One such condition is that the point lies in the interior of the state constraint set. Further, when global optima lie on the boundary, we propose a regularization-based modification that guarantees existence of a closed-loop equilibrium in the interior, of arbitrarily small suboptimality;
    \item We prove local asymptotic stability of local optima in the interior of the state constraint set. Under additional assumptions, in cases with a unique global optimum, we show global asymptotic stability of the global optimum.
\end{itemize}

\subsubsection*{Related work}
SGFs were originally introduced as autonomous dynamical systems that solve constrained optimization problems \cite{allibhoy2023control}. SGFs have been used in reinforcement learning~\cite{JF-WC-JC-YS:23-csl,PM-AM-JC:25-l4dc} and neural network-based approximations~\cite{DA-YC-AC-JC-EDA:25-tsg} to ensure safety. As controllers interconnected with systems, SGFs have been employed for feedback optimization in \cite{YC-LC-JC-EDA:23-csl}. However, \cite{YC-LC-JC-EDA:23-csl} only guarantees anytime satisfaction of input constraints. Towards enforcing state constraints, our work is the first to combine SGFs with high-order CBFs, which significantly complicates technical analysis. High-order CBFs are key to enforcing state constraints, since these are of high relative degree with respect to the controller dynamics, and thus cannot be handled by standard SGFs or projected gradient flows (PGFs) interconnected to a plant, as these dynamics manipulate only the controller's vector-field. Finally, note that SGFs are continuous approximations of PGFs \cite{delimpaltadakis2023relationship,delimpaltadakis2024continuous}, which have also been used for feedback optimization with input constraints \cite{hauswirth2021optimization}. 

Existing works on feedback optimization either disregard state constraints completely, enforce them only asymptotically, e.g., \cite{GB-JC-JIP-EDA:22-tcns, MC-ED-AB:20, brunner2012feedback, LSPL-JWSP-EM:24, YC-LC-JC-EDA:23-csl}, or consider only \emph{static plants}, e.g., \cite{haberle2020non,colot2024optimal}. The latter two cases can be straightforwardly handled by enforcing corresponding input constraints, as at the steady state, or for static plants, there is a one-to-one correspondence between input and state. However, these techniques typically lead to state violations during transients, as showcased in Section \ref{sec:simulations}. In contrast, our method enforces state constraints at all times, while retaining equivalence between critical points and equilibria, and mediating safety and timescale separation, guaranteeing asymptotic stability. Nonetheless, compared to other works (e.g., \cite{hauswirth2020timescale,GB-MV-JC-EDA:24-tac}), it assumes complete knowledge of the plant dynamics.\footnote{Extending to model-free scenarios is left for future work.}

Related to feedback optimization are the techniques of extremum seeking (ES; e.g. \cite{chen2025continuous,williams2024semiglobal}) and online convex optimization (OCO; e.g. \cite{li2021online,zhou2023safe, karapetyan2023online, nonhoff2022online}). Like the aforementioned feedback-optimization works, ES methods, which address state constraints, either consider static plants or (implicitly) assume that the plant operates at the steady state; see \cite{chen2025continuous, williams2024semiglobal} and references therein. For a comparison between feedback optimization and ES, see \cite{hauswirth2021optimization}. Regarding OCO, the main differences are that OCO generally considers discrete-time plants and focuses on finite-horizon regret guarantees.
Instead, in this work we consider continuous-time plants and study asymptotic convergence to the optimizer.

Finally, a preliminary, conference version of this article has been presented in \cite{delimpaltadakis2025feedback}. The present work extends \cite{delimpaltadakis2025feedback} by: a) proving global asymptotic stability in certain scenarios with unique optimizers and convex objective functions (Theorem~\ref{thm:global_convergence}; \cite{delimpaltadakis2025feedback} provided only local asymptotic stability results), b) providing formal statements and proofs for Proposition~\ref{prop:crcq} and Proposition~\ref{prop:regularization}, which were only alluded to as remarks in \cite{delimpaltadakis2025feedback}, and c) providing a complete technical treatment of all theoretical results (proofs were omitted in~\cite{delimpaltadakis2025feedback}).

\section{Problem statement}
Consider\footnote{For $r\in\mathbb{Z}_{>0}$, we denote $[r]:=\{ 1, 2, \hdots, r \}$. We adopt the following conventions on dimensions: given $g:\real^n\times\real^m\to\real^k$, $(x,u)\mapsto g(x,u)$, continuously differentiable, then $\dfrac{\partial g}{\partial x}(x,u)\in\real^{k\times n}$; given $g:\real^n\to\real$, $x\mapsto g(x)$, continuously differentiable, then $\nabla g(x) = (\frac{\partial g}{\partial x}(x))^\top\in\real^n$. 
Given the optimization problem 
\begin{equation*}
    \min_x \Phi(x) \ \text{s.t.} \ a_i(x)\geq 0, \ b_j(x) = 0, \ i\in[p], \ j\in[q],
\end{equation*}
where $\Phi,a_i,b_j:\real^n\to\real$ are continuously differentiable,
a KKT (or critical) point is a point $\bar{x}\in\real^n$ for which there exist $\bar{\lambda}\in\real^p$, $\bar{\mu}\in\real^q$ satisfying 
\begin{align*}
    &\nabla \Phi (\bar{x}) - \sum_{i=1}^p \bar{\lambda}_i \nabla a_i(\bar{x}) + \sum_{j=1}^q \bar{\mu}_j \nabla b_j(\bar{x}) = 0, \ \bar{\lambda}_i \geq 0, \ i\in[p], \\
    &\bar{\lambda}_i a_i(\bar{x}) = 0, \ a_i(\bar{x}) \geq 0, \ b_j(\bar{x}) = 0, \ i\in[p], \ j\in[q].
\end{align*}
Furthermore, if there exists a neighborhood $\Nc\subseteq\{x:a_i(x)\geq 0, \ b_j(x)=0, \ i\in [p], \ j\in [q]\}$ of $\bar{x}$ such that, for all $x\in\Nc$, $\Phi(\bar{x}) \leq \Phi(x)$, we say that $\bar{x}$ is a local optimum.
If $\Phi(\bar{x}) \leq \Phi(x)$ for all $x\in \{x:a_i(x)\geq 0, \ b_j(x)=0, \ i\in [p], \ j\in [q]\}$, then $\bar{x}$ is a global optimum.

Now, consider the parametric optimization problem
\begin{equation}\label{eq:parametric_nlp}
    \min_x \Phi(x,\theta) \ \text{s.t.} \ a_i(x,\theta)\geq 0, \ i\in[p], \ b_j(x,\theta)=0, \ j\in[q],
\end{equation}
where $\Phi,a_i,b_j:\real^n\times\real^l\to\real$ are continuously differentiable and $\theta\in\real^l$ is a parameter.
Denote $J(x,\theta):=\{i\in[p]: a_i(x,\theta)=0\}$. The following conditions are commonly used in parametric optimization for \emph{sensitivity analysis} \cite{JL:95}, although often more relaxed conditions suffice:
\begin{itemize}
    \item \emph{Mangasarian-Fromovitz Constraint Qualification} (MFCQ): Problem \eqref{eq:parametric_nlp} satisfies MFCQ at $(x_*,\theta_*)$ if there exists $\psi \in\real^n$ such that $\frac{\partial a_i}{\partial x}(x_*,\theta_*) \psi > 0$, for all $i\in J(x_*,\theta_*)$ and $\frac{\partial b_j}{\partial x}(x_*,\theta_*) \psi = 0$ for all $j\in[q]$;
    \item \emph{Constant Rank Constraint Qualification} (CRCQ): Problem \eqref{eq:parametric_nlp} satisfies CRCQ at $(x_*,\theta_*)$ if 
    for any subset $\tilde{J} \subset J(x_*,\theta_*)$, $\{\frac{\partial a_i}{\partial x}(x,\theta)\}_{i\in \tilde{J}} \cup \{ \frac{\partial b_j}{\partial x}(x,\theta) \}_{j=1}^p$ remains of constant rank in a neighborhood of $(x_*,\theta_*)$.
    \item \emph{Linear Independence Constraint Qualification} (LICQ): Problem \eqref{eq:parametric_nlp} satisfies the LICQ at $(x_*,\theta_*)$ if $\{\frac{\partial a_i}{\partial x}(x,\theta)\}_{i\in J(x,\theta)} \cup \{ \frac{\partial b_j}{\partial x}(x,\theta) \}_{j=1}^p$ is linearly independent at $(x_*,\theta_*)$.
\end{itemize}
For the problem~\eqref{eq:parametric_nlp}, if the CRCQ or MFCQ holds at $(x_*,\theta_*)$, the KKT conditions at $\theta = \theta^*$ are necessary for local optimality of $x_*$, for $\theta=\theta_*$.} the control system
\begin{align}\label{eq:system}
    \dot{\xi}(t) = f(\xi(t),\upsilon(t)),
\end{align}
where $f:\real^n\times\real^m\to\real^n$ is a vector field, and $\xi(t)$, $\upsilon(t)$ are the state and control variables at time $t \in \real_{\geq0}$, respectively.

\begin{assumption}\longthmtitle{Globally exponentially stable equilibrium}\label{assum:stability}
    For any constant input signal 
    $\upsilon(t) = u$,
    system \eqref{eq:system} admits a unique globally exponentially stable equilibrium. 
\end{assumption}

Assumption \ref{assum:stability} is common in the feedback-optimization literature~\cite{hauswirth2021optimization}. It implies that, if the controller dynamics is slow enough compared to the plant, then the interconnection is stable. Let  $w:\real^m\to\real^n$ be the map that assigns each input $u\in\real^m$ to the unique equilibrium $w(u)$, i.e.,  $f(w(u),u) = 0$.

In feedback optimization, the goal is to design a controller $\upsilon$
that drives system \eqref{eq:system} to a steady-state state-input pair $(x_*,u_*)$, which solves an optimization problem:
\begin{equation}\label{eq:steady-state_opti_problem}
    (x_*,u_*)\in\left\{\begin{aligned}\argmin_{x,u} \text{ }&\Phi(x)\\
            \mathrm{s.t.} \text{ } & x=w(u), \, h(x)\geq 0, \, b(u)\geq 0,\end{aligned}
            \right.
\end{equation}
where $\Phi:\real^n\to\real$, $h:\real^n\to\real$ and $b:\real^m\to\real$ are sufficiently smooth. The equality constraint $x = w(u)$ guarantees that $(x_*,u_*)$ is an admissible steady state for \eqref{eq:system}. In this work, we also require that the input and state inequality constraints are satisfied at all times, i.e., for all $t \in \real_{\geq0}$:
\begin{equation}\label{eq:transient_constraints}
    \begin{aligned}
        (\xi(t),\upsilon(t)) \in \Sc:=\{(x,u)\in\real^n\times\real^m: \ &h(x)\geq 0, \\ &b(u) \geq 0\}.
    \end{aligned}    
\end{equation}
We aim to design a controller $\upsilon$ that drives~\eqref{eq:system} to a steady-state that is a KKT point of~\eqref{eq:steady-state_opti_problem} while satisfying \eqref{eq:transient_constraints}.
The above problem, \emph{without the transient state constraints $h(\xi(t))\geq 0$}, has been solved by designing a dynamic controller, with dynamics given by, e.g., a PGF \cite{hauswirth2021optimization, hauswirth2020timescale} or an SGF \cite{YC-LC-JC-EDA:23-csl}.

\section{Feedback optimization with state constraints}
In this section we introduce our control design, which combines the safe gradient flow~\cite{allibhoy2023control} and high-order control barrier functions (HOCBFs)~\cite{xiao2021high}. Before presenting it, we introduce necessary  notation related to HOCBFs.

We denote $\lie_f h(x,u) = \frac{\partial h}{\partial x}(x) \cdot f(x,u)$ and, for convenience, $\lie^i_f h(x,u) = \frac{\partial \lie^{i-1}_fh}{\partial x} \cdot f(x,u)$, for $i\geq 2$.
We further define $h_i:\real^n\times\real^m\to\real$ for $i\in\mathbb{Z}_{\geq0}$ as:
    \begin{align}\label{eq:h_i}
        h_0(x,u) &:= h(x), \\
        h_i(x,u) &:= \dfrac{\partial h_{i-1}}{\partial x}(x,u) \cdot f(x,u) + \beta h_{i-1}(x,u), \quad i \geq 1, \notag
    \end{align}
where $\beta>0$ is a design parameter. In what follows, we consider $\beta>0$ fixed. Let the sets $\Sc_i:=\{(x,u)\in\real^n\times\real^m: b(u)\geq 0, h_i(x,u)\geq 0\}$ for $i\in\mathbb{Z}_{\geq 0}$.
Note that $\Sc_0=\Sc$. For systems \eqref{eq:system} with relative degree $r\in\mathbb{N}$ (cf. Assumption~\ref{assum:differentiability, regularity and relative degree} for a precise definition),
the functions $\{ h_i \}_{i=0}^{r-1}$ are independent of $u$, while $h_r$ depends on $u$. In \cite{xiao2021high}, $\upsilon$ is designed such that $h_r(\xi(t),\upsilon(t))\geq 0$ for all $t$, which renders $\bigcap_{i=0}^{r} \Sc_i$ forward invariant, due to the definition of the functions $h_i$.

\begin{remark}\longthmtitle{Effect of $\beta$ on $\Sc_i$}\label{rem:tuning beta}
    As $\beta$ increases, the sets $\{ \Sc_i \}_{i=1}^{r-1}$ approximate $\Sc$ more closely. Informally, when $\Sc$ is compact and $f$ and $h$ are sufficiently smooth, we can take $\beta$ sufficiently large, so that $h_i(x,u)\approx \beta^i h(x,u)$ for $i\in[m]$.
\end{remark}

We also make the following assumptions:
\begin{assumption}\longthmtitle{Relative degree and regularity}\label{assum:differentiability, regularity and relative degree}
    \begin{enumerate}
        \item\label{it:diff-regularity-first} There exists $r\in\mathbb{Z}_{>0}$ such that:
        \begin{itemize}
            \item the functions $f$ and $h$ are $(r+1)$-times and $(r+2)$-times continuously differentiable, respectively;
            \item $\Sc_r \neq \emptyset$; $\frac{\partial h_r(x,u)}{\partial u}(x,u)\neq 0$, for all $(x,u)\in\Sc_r$ such that $h_r(x,u)=0$; and $\frac{\partial h_i(x,u)}{\partial u} = 0$
            for all $(x,u)\in\Sc_r$ and all $i<r$.
        \end{itemize} \label{assum_item:relative degree}
        \item\label{it:diff-regularity-second} The function $b$ is twice continuously differentiable, and $b(u) = 0 \implies \nabla b(u) \neq 0$, for all $(x,u)\in \Sc$.
        \item $\Phi$ and $w$ are twice continuously differentiable.
    \end{enumerate}
\end{assumption}
In what follows, item \ref{assum_item:relative degree} is proven key to design feedback-optimization controllers that enforce state constraints. It implies that the control input appears, after differentiating $h$ along the system dynamics $r$-times. Further, it implies that all functions $\{h_i\}_{i<r}$ depend only on~$x$, and the function $h_r$ is an $r$-order CBF for system \eqref{eq:system}, cf.~\cite{xiao2021high}. 

Towards solving the feedback-optimization problem with state constraints, we propose the controller $\dot{\upsilon}(t) = g_{\epsilon,\alpha,\gamma}(\xi(t),\upsilon(t))$, where
\begin{equation}\label{eq:sgf_QP}
\begin{aligned}
    &g_{\epsilon,\alpha,\gamma}(x,u)=\\
    &\qquad\left\{\begin{aligned}\argmin_{q}\text{ }& \frac{1}{2}\|q +\epsilon\frac{\partial w}{\partial u}(u)^\top \nabla\Phi(x)\|^2 \\
            \mathrm{s.t.:} \text{ } &\nabla^\top b(u)\cdot q +\alpha b(u)\geq 0\\
            &\begin{aligned}\frac{\partial h_r}{\partial x}(x,u)\cdot f(x,u)+ \frac{\partial h_r}{\partial u}(x,u)\cdot q&
            \\+\gamma h_{r}(x,u)\geq 0&\end{aligned}
            \end{aligned}\right.
\end{aligned}
\tag{SGF}
\end{equation}
and $\epsilon>0$, $\alpha>0$ and $\gamma>0$ are parameters to be designed. The closed-loop dynamics become:
\begin{equation}
    \begin{aligned}
        \dot{\xi}(t) &= f(\xi(t),\upsilon(t)), \\
        \dot{\upsilon}(t) &= g_{\epsilon,\alpha,\gamma}(\xi(t),\upsilon(t)).
    \end{aligned}
    \label{eq:closed_loop_sgf}
    \tag{SGF-CL}
\end{equation}
We refer to the controller $g_{\epsilon,\alpha,\gamma}$ as the safe gradient flow (SGF).\footnote{Note that, when evaluated at a steady-state $(w(u),u)$, $g_{\epsilon,\alpha,\gamma}(w(u),u)$ is equivalent to the safe gradient flow vector field, as defined in~\cite{allibhoy2023control}, associated with the steady-state optimization problem $\min_{\bar{u}}\Phi(w(\bar{u}))$ s.t. $b(\bar{u})\geq0, \ h(w(\bar{u}))\geq 0$, which is equivalent to the original problem \eqref{eq:steady-state_opti_problem}.}
Let us provide the intuition behind its design. It is known that the standard gradient-flow controller $\dot{\upsilon} = -\epsilon\frac{\partial w}{\partial u}(\upsilon)^\top \nabla\Phi(\xi)$ drives the system \eqref{eq:system} to the critical points of the feedback-optimization problem \eqref{eq:steady-state_opti_problem}, without the inequality constraints. The controller \eqref{eq:sgf_QP}, inspired by SGFs \cite{allibhoy2023control}, modifies the standard gradient flow in a point-wise minimal manner through a QP, so that the two inequality constraints in \eqref{eq:sgf_QP} are enforced. As will be shown later, these conditions, enforce that $b(\upsilon(t))\geq 0$ and $h(\xi(t))\geq 0$ for all $t\in \mathbb{R}_{\geq0}$, in turn guaranteeing forward invariance of the subset $\bigcap_{i=0}^r \Sc_i$ of $\Sc$, thus enforcing state and input constraints. The QP has a unique solution and can easily be solved online, although the present work does not focus on implementation considerations.

\begin{remark}[Enforcing state constraints]
    Inspired by standard CBF techniques \cite{ames2019control_review}, to enforce state constraints one could place a so-called safety filter in series interconnection with a standard gradient flow. However, this modifies the controller's state directly. This non-trivial modification to the gradient flow complicates analysis on stability of the optima of optimization problem \eqref{eq:steady-state_opti_problem} and on equivalence between the closed-loop's equilibria and the optimization problem's critical points. In contrast, our method, employing high-order CBFs, modifying the controller dynamics (not the state), enables such results.
\end{remark}

In the coming sections, we derive a number of results on the proposed controller~\eqref{eq:sgf_QP}. The proofs are presented in Section \ref{sec:proofs}, unless otherwise stated.

\section{Well-posedness and safety results}\label{sec:safety and well-posedness}
Here, we establish the well-posedness of~\eqref{eq:closed_loop_sgf} and its anytime satisfaction of state and input constraints.

\subsection{Well-posedness}\label{subsec:well-posedness}
To guarantee well-posedness of  \eqref{eq:closed_loop_sgf}, we present results on the feasibility of the QP in \eqref{eq:sgf_QP} and on the existence and uniqueness of solutions of \eqref{eq:closed_loop_sgf}. 

\begin{proposition}[Feasibility]\label{prop:feasibility}
    Let Assumption \ref{assum:differentiability, regularity and relative degree} hold. Assume that, for some compact set $\Cc_r \subseteq \Sc_r$, for any $(x,u) \in \Cc_r$,  there exists $q_{(x,u)}\in\real^m$ such that:
    \begin{itemize}
        \item if $b(u)=0$, then $\nabla^\top b(u) q_{(x,u)} > 0$;
        \item if $h_r(x,u)=0$, then $\frac{\partial h_r}{\partial u}(x,u)q_{(x,u)} + \frac{\partial h_r}{\partial x}(x,u)f(x,u)> 0$.
    \end{itemize}
    Then, there exist $\alpha_f, \gamma_f > 0$ such that, for all $\alpha > \alpha_f$, $\gamma > \gamma_f$, the QP in \eqref{eq:sgf_QP} is feasible for any $(x,u)\in\Cc_r$.
\end{proposition}
 
Let us provide the intuition of the assumptions above. Comparing the above inequalities with MFCQ for the set $\Sc_r$, the only difference is the presence of $\frac{\partial h_r}{\partial x}(x,u)f(x,u)$. This term is present here, as the controller accounts for external plant dynamics, to keep trajectories in $\Sc_r$. In standard optimization, this term is absent, as there is no external plant influencing the trajectories of the optimization algorithm (the controller, in our case). Hence, the above inequalities require sufficient regularity of the constraint set $\Sc_r$ (MFCQ), while accounting for plant dynamics (the term $\frac{\partial h_r}{\partial x}(x,u)f(x,u)$). Finally, compactness is required for $\alpha_f$ and $\gamma_f$ to be finite. 

\begin{proposition}[Existence and uniqueness of solutions of \eqref{eq:closed_loop_sgf}]\label{prop:solutions_sgf}
Under the assumptions of, and with the same notation as Proposition~\ref{prop:feasibility}, further assume that CRCQ holds for each $(x,u)\in\Cc_r$ for~\eqref{eq:sgf_QP}. Then, for $\alpha > \alpha_f$, $\gamma > \gamma_f$, $g_{\epsilon,\alpha,\gamma}$ is locally Lipschitz in $\Cc_r$. Thus, \eqref{eq:closed_loop_sgf} has a unique solution for every initial condition in~$\Cc_r$.
\end{proposition}

The next result provides a condition that guarantees the satisfaction of CRCQ.
\begin{proposition}\longthmtitle{CRCQ satisfaction}\label{prop:crcq}
    Consider a compact subset $\Cc_r\subseteq\Sc_r$. Let Assumption \ref{assum:differentiability, regularity and relative degree} hold. Assume that for all $(x,u)\in\Cc_r$, if $b(u) = 0$ and $h_r(x,u) = 0$, then $\nabla b(u)$ and $\frac{\partial h_r}{\partial u}(x,u)$ are linearly independent. Then, there exist $\alpha_0$, $\beta_0$ such that by taking $\alpha > \alpha_0$, $\gamma > \gamma_0$,~\eqref{eq:sgf_QP} satisfies the CRCQ for all $(x,u)\in\Cc_r$.
\end{proposition}

The conditions in Proposition~\ref{prop:crcq} are not equivalent to LICQ for~\eqref{eq:sgf_QP}, as the points where the constraints of~\eqref{eq:sgf_QP} are active are not necessarily those where $b(u) = 0$ or $h_r(x,u) = 0$ (see~\cite{PM-AA-JC:24-ejc} for alternative conditions guaranteeing local Lipschitzness of optimization-based controllers).

\subsection{Safety}\label{subsec:safety}
Let us show how the proposed controller enforces state and input constraints at all times. We refer to constraint satisfaction as \emph{safety}, which is common in the CBF literature~\cite{ames2019control_review}.

\begin{proposition}[Safety]\label{prop:safety}
Let Assumption \ref{assum:differentiability, regularity and relative degree} hold.
If the QP in \eqref{eq:sgf_QP} is feasible everywhere in $\bigcap_{i=0}^r\Sc_{i}$, and \eqref{eq:closed_loop_sgf} has a unique solution for any initial condition in $\bigcap_{i=0}^r\Sc_{i}$, then $\bigcap_{i=0}^r\Sc_{i}$ is forward invariant for the closed loop \eqref{eq:closed_loop_sgf}.
\end{proposition} 

\begin{remark}[Compactness]
    Props. \ref{prop:feasibility} and \ref{prop:solutions_sgf} provide sufficient conditions to meet the conditions in Proposition~\ref{prop:safety}. Notice that it suffices that one of the sets $\{ \Sc_i \}_{i=0}^r$ is compact, so that one may choose $\Cc_r:=\bigcap_{i=0}^{r}\Sc_{i} \subseteq \Sc_r$ to be compact. 
\end{remark}

Forward invariance of $\bigcap_{i=0}^{r}\Sc_{i}$ implies that, if the initial condition lies in $\bigcap_{i=0}^{r}\Sc_{i}$, then the proposed controller enforces state and input constraints at all times, i.e., $b(\upsilon(t))\geq0$ and $h(\xi(t))\geq 0$, for all $t\in\mathbb{R}_{\geq0}$.
Indeed, this is because $\bigcap_{i=0}^{r}\Sc_{i} \subset \Sc_0$.
Recall also from Remark~\ref{rem:tuning beta} that, under compactness, sufficiently large $\beta$ yields $\bigcap_{i=0}^{r}\Sc_{i} \approx \Sc$, and hence we can make the difference between the sets $\bigcap_{i=0}^{r}\Sc_{i}$ and $\Sc$ arbitrarily small. Importantly, as the following result shows, the forward-invariant set $\bigcap_{i=0}^{r}\Sc_{i}$ contains all feasible points of the feedback-optimization problem~\eqref{eq:steady-state_opti_problem}, and thus all desired steady states (critical points and optima).

\begin{proposition}\longthmtitle{Feasible points of~\eqref{eq:steady-state_opti_problem} lie in $\bigcap_{i=0}^r\Sc_i$}\label{prop:feasible-points-lie-in-bigcap}
    Let $(x_*,u_*)$ be a feasible point of~\eqref{eq:steady-state_opti_problem}. Then, $(x_*,u_*)\in\bigcap_{i=0}^r\Sc_i$.
\end{proposition}

\section{Critical points, equilibria and convergence to optima}\label{sec:critical_points_convergence}
Here, we establish the equivalence between critical points of the optimization problem and equilibria of the closed-loop system, and the asymptotic stability of optima.

\subsection{Relationship between critical points and equilibria}\label{subsec:cirical_points_equilibria_regularization}
Let us unveil the relationship between the equilibria of the closed loop  \eqref{eq:closed_loop_sgf} and critical points of \eqref{eq:steady-state_opti_problem}.

\begin{proposition}[Critical points of \eqref{eq:steady-state_opti_problem} - equilibria of \eqref{eq:closed_loop_sgf}]\label{prop:critical_points_equilibria}
    Assume that either MFCQ or CRCQ holds at $(x_*,u_*)\in\Sc$ for optimization problem \eqref{eq:steady-state_opti_problem} and the QP in~\eqref{eq:sgf_QP}\footnote{As the QP \eqref{eq:sgf_QP} has linear constraints, the satisfaction of MFCQ or CRCQ is independent of the value of the decision variable~$q$.}, and that at least one of the following holds:
    \begin{enumerate}
        \item $h(x_*) >0$;
        \item $\nabla h(x_*)$ is a right eigenvector of $\frac{\partial f}{\partial x}(x_*,u_*)$ with eigenvalue $e$ and $e^r<0$;
        \item $(\frac{\partial w}{\partial u}(u_*))^\top$ is a left eigenvector of $\frac{\partial f}{\partial x}(x_*,u_*)$ with eigenvalue $e$ and $e^r<0$.
    \end{enumerate}
    Then $(x_*,u_*)$ is a critical point of \eqref{eq:steady-state_opti_problem} if and only if it is an equilibrium of \eqref{eq:closed_loop_sgf}.
\end{proposition}

Condition 1 in Proposition~\ref{prop:critical_points_equilibria} establishes that, in the interior of the state constraint set $\{x:\text{ }h(x)\geq0\}$, all equilibria of \eqref{eq:closed_loop_sgf} are critical points of the feedback-optimization problem \eqref{eq:steady-state_opti_problem}, and vice versa. In practice, this is not restrictive, as operating right on the safety boundary $h(x)=0$ is generally undesirable. Nevertheless, equivalence between critical points and equilibria on the boundary is guaranteed if condition 2 or 3 in Proposition~\ref{prop:critical_points_equilibria} holds. If these conditions are not satisfied, however, the closed loop might have spurious equilibria on the boundary, which are not critical points. 

When a global optimum of \eqref{eq:steady-state_opti_problem} satisfies $h(x_*)=0$, the following results show how to modify the controller \eqref{eq:sgf_QP} to obtain a corresponding equilibrium in the interior of the state constraint set that is of arbitrarily small suboptimality.

\begin{proposition}[Regularization of the feedback optimization problem]\label{prop:regularization}
    Let Assumption \ref{assum:differentiability, regularity and relative degree} hold. Denote by $(x_*,u_*)$ a global optimum of \eqref{eq:steady-state_opti_problem} and assume that $h(x_*)=0$ and MFCQ holds for~\eqref{eq:steady-state_opti_problem} at $(x_*,u_*)$. Given $p,\varepsilon >0$, consider: 
    \begin{align}\label{eq:perturbed-optimization-problem}
        \notag
        \min\limits_{x,u}& \ \Phi(x) + p (\varepsilon-h(x))^2 \\
        \mathrm{s.t.}& \ x=w(u), \ b(u) \geq 0, \ h(x) \geq 0.
    \end{align}
    Then, for any $c>0$, there exist $p,\varepsilon>0$ such that the optimizer $(x_*',u_*')$ of~\eqref{eq:perturbed-optimization-problem} satisfies $h(x_*') > 0$ and $|\Phi( x_*')-\Phi(x_*)| < c$.
\end{proposition}

The next result is a direct byproduct of Props.~\ref{prop:critical_points_equilibria} and~\ref{prop:regularization}.

\begin{corollary}[Regularization of the controller design]\label{cor:regularization}
    Under the assumptions and notation of Proposition~\ref{prop:regularization}, further assume that MFCQ holds for~\eqref{eq:sgf_QP} at $(x_*,u_*)$. For any $c>0$, there exist $p,\varepsilon>0$ such that, if we replace $\Phi(x)$ with $\Phi(x) + p (\varepsilon-h(x))^2$ in \eqref{eq:sgf_QP}, then \eqref{eq:closed_loop_sgf} has an equilibrium $(x',u')$ with $h(x')>0$ and $|\Phi( x')-\Phi(x_*)| < c$.
\end{corollary}

\subsection{Asymptotic stability of optima in the interior of safe set}
Here, we provide conditions under which the optima of the 
problem \eqref{eq:steady-state_opti_problem} that lie in the interior of the state-constraint set are asymptotically stable. First, we recall the following result, which follows from the open-loop stability of system \eqref{eq:system}.
\begin{lemma}[\hspace{-.1mm}{\cite[Lemma 2.1]{YC-LC-JC-EDA:23-csl}}]\label{lem:W}
Let Assumption \ref{assum:stability} hold. Then, there exist a function $W:\real^n\times\real^m\to\real$ and positive constants $d_1, d_2, d_3, d_4$, such that, for any $(x,u)\in\real^{n+m}$:
    \begin{align*}
        &d_1 \norm{x-w(u)}^2 \leq W(x,u) \leq d_2 \norm{x-w(u)}^2, \\
        &\frac{\partial W(x,u)}{\partial x}f(x,u) \leq -d_3 \norm{x-w(u)}^2, \\
        &\norm{\frac{\partial W(x,u)}{\partial u}} \leq d_4 \norm{x-w(u)}.
    \end{align*}
\end{lemma}

Now, we first show that the equilibria of the closed loop in the interior of the state constraint set, which correspond to optima of \eqref{eq:steady-state_opti_problem}, can be rendered locally asymptotically stable.

\begin{theorem}[Local asymptotic stability of local optima]\label{thm:local_convergence}
    Let Assumptions \ref{assum:stability} and \ref{assum:differentiability, regularity and relative degree} hold. Let $(x_*,u_*)$ be a local optimum of \eqref{eq:steady-state_opti_problem} with $h(x_*)>0$ and $\Nc$ a neighborhood where $(x_*,u_*)$ is the only critical point. Assume that $u\mapsto \Phi(w(u))$ has compact level sets. Assume the QP in~\eqref{eq:sgf_QP} is feasible everywhere in $\Nc$, that MFCQ or CRCQ holds in $\Nc$ for \eqref{eq:steady-state_opti_problem} and the QP in~\eqref{eq:sgf_QP}, and that unique solutions of~\eqref{eq:closed_loop_sgf} exist for any initial condition in $\Nc$. Then, there exist $\epsilon_*,\gamma_*>0$, such that, for any $\epsilon \in (0,\epsilon_*)$ and $\gamma \geq \gamma_*$, the point $(x_*,u_*)$ is locally asymptotically stable relative\footnote{This means that there exist neighborhoods $\Nc_1, \Nc_2$ of $(x_*,u_*)$ such that the trajectories with initial condition in $\Nc_1\cap\Sc_r$ stay in $\Nc_2\cap\Sc_r$ for all $t\in\real_{\geq0}$ and converge to $(x_*,u_*)$ as $t\to\infty$.} to~$\Sc_r$ for~\eqref{eq:closed_loop_sgf}.
\end{theorem}

Beyond local asymptotic stability, when \eqref{eq:steady-state_opti_problem} admits a unique optimum, we obtain global asymptotic stability, under additional conditions. Let us first introduce a technical result, which is used in the sequel.

\begin{lemma}\longthmtitle{Conditions for $\sigma$-strong MFCQ}\label{lem:conditions-sigma-strong-mfcq}
   Let the assumptions of Proposition~\ref{prop:feasibility} hold.
   Then, for any compact subset $\Cc_r\subseteq \Sc_r$, there exists $\sigma \in(0,1)$ such that 
   for any $(x,u)\in\Cc_r$, there exists $q_{(x,u)}\in\real^m$ such that 
   \begin{align}\label{eq:sigma-strong-mfcq}
       &\bullet \text{ if $b(u) = 0$, then $\nabla b(u)^\top q_{(x,u)} \geq \sigma \|q_{(x,u)}\|\|\nabla b(u)\|$;} \notag \\
       &\bullet \text{ if $h_r(x,u)=0$, then} \notag\\
       &\frac{\partial h_r}{\partial u}(x,u)q_{(x,u)} + \frac{\partial h_r}{\partial x}(x,u)f(x,u)\geq\\&\qquad\qquad\qquad\qquad\qquad\sigma \norm{(f(x,u), \ q_{(x,u)})}\norm{\nabla h_r(x,u)}. \notag
   \end{align}
\end{lemma}
We are now ready to state conditions for global asymptotic stability of optima.
\begin{theorem}[Global asymptotic stability]\label{thm:global_convergence}
    Let Assumptions \ref{assum:stability} and \ref{assum:differentiability, regularity and relative degree} hold, suppose $\bigcap_{i=0}^r\Sc_i$ is compact, and the assumptions of Proposition~\ref{prop:feasibility} hold over $\bigcap_{i=0}^r\Sc_i$. Assume that for all $(x,u)\in\bigcap_{i=0}^r\Sc_i$, if $b(u) = 0$ and $h_r(x,u) = 0$, then $\nabla b(u)$ and $\frac{\partial h_r}{\partial u}(x,u)$ are linearly independent. Further, assume:
    \begin{enumerate}
        \item The function $\Phi(w(\cdot))$ 
        is $c$-strongly convex~\footnote{A function $\psi:\real^n\to\real$ is $c$-strongly convex if for all $x,y\in\real^n$, $\psi(y) \geq \psi(x) + \nabla\psi(x)^\top (y-x) + \frac{c}{2}\norm{y-x}^2$.}, with $c>0$;
        \item The functions $-b(\cdot)$ and $-h_r(x,\cdot)$ are convex (the latter for any $x$);
        \item\label{it:weird-assumption} There exists a unique optimizer $(x_*,u_*)$ to the feedback-optimization problem \eqref{eq:steady-state_opti_problem}
        and it satisfies $b(u_*)>0$ and $h_r(x,u_*)>0$, for any $x$ for which there exists $u$ with $(x,u)\in\bigcap_{i=0}^r\Sc_i$; 
        \item\label{it:d_3-d_4} $d_3 - d_4(1+\frac{1}{\sigma}) l_f > 0$, where $l_f$ is the Lipschitz constant of $f$ on $\bigcap_{i=0}^r\Sc_i$, $d_3,d_4>0$ are constants associated to the open loop \eqref{eq:system}, as dictated by Lemma \ref{lem:W}, and $\sigma$ depends on $f$ and $\Sc_r$, as dictated by Lemma \ref{lem:conditions-sigma-strong-mfcq}.
    \end{enumerate}
    Then, there exist $\alpha_*,\gamma_*,\epsilon_*>0$, such that, for any $\alpha > \alpha_*$, $\gamma > \gamma_*$ and $\epsilon \in (0,\epsilon_*)$, the global optimum $(x_*,u_*)$ is globally asymptotically stable relative to $\bigcap_{i=0}^r\Sc_i$ for \eqref{eq:closed_loop_sgf}.
\end{theorem}

Let us briefly discuss the assumptions of Theorem~\ref{thm:global_convergence}. 
The convexity assumptions in items 1 and 2 are standard in practice. E.g., provided that $\Phi$ is convex and the dynamics \eqref{eq:system} are linear $\Phi(w(\cdot))$ is convex. A common case in which $-h_r(x,\cdot)$ is convex for all $x$ is if the dynamics are control-affine, since in this case, $-h_r(x,\cdot)$ is affine. Furthermore, one simple setting in which there exists a unique optimizer $(x_*, u_*)$ to~\eqref{eq:steady-state_opti_problem} (as assumed in item 3) is if, in addition to the convexity assumptions in items 1 and 2, $-h$ is convex. Indeed, in this case the problem~\eqref{eq:steady-state_opti_problem} is convex when the state variable $x$ is eliminated. 
Item~\ref{it:weird-assumption} allows for the use of convexity arguments in the proof of Theorem~\ref{thm:global_convergence}. Simulations in Section~\ref{sec:discussion-assumption-global-convergence} indicate that in certain cases, it is not necessary for global stability. Item~\ref{it:d_3-d_4} requires that the open-loop stability of system~\eqref{eq:system}, guaranteed by Assumption~\ref{assum:stability} (and quantified by $d_3$), dominates the effects of potentially destabilizing inputs (quantified by $d_4(1+\frac{1}{\sigma})l_f$), that are necessary to enforce safety. Finally, note that if the global optimum is on the boundary $h(x_*)=0$, one may employ Corollary~\ref{cor:regularization} and Theorem~\ref{thm:local_convergence} (resp., Theorem~\ref{thm:global_convergence}) to render a slightly suboptimal point locally (resp. globally) asymptotically stable.

\section{Simulations}\label{sec:simulations}
We implement the proposed controller and illustrate the theoretical results on safety, convergence to local optima, and stability. We also compare our approach with similar techniques from the state-of-the-art, discuss the necessity of the assumptions used for establishing global asymptotic stability, and illustrate the regularization procedure described in Proposition~\ref{prop:regularization}.
We consider convex and non-convex problems~\eqref{eq:steady-state_opti_problem}.

\begin{figure*}[tbh]
        \centering
        \subfigure[]{\includegraphics[height=.24\linewidth]{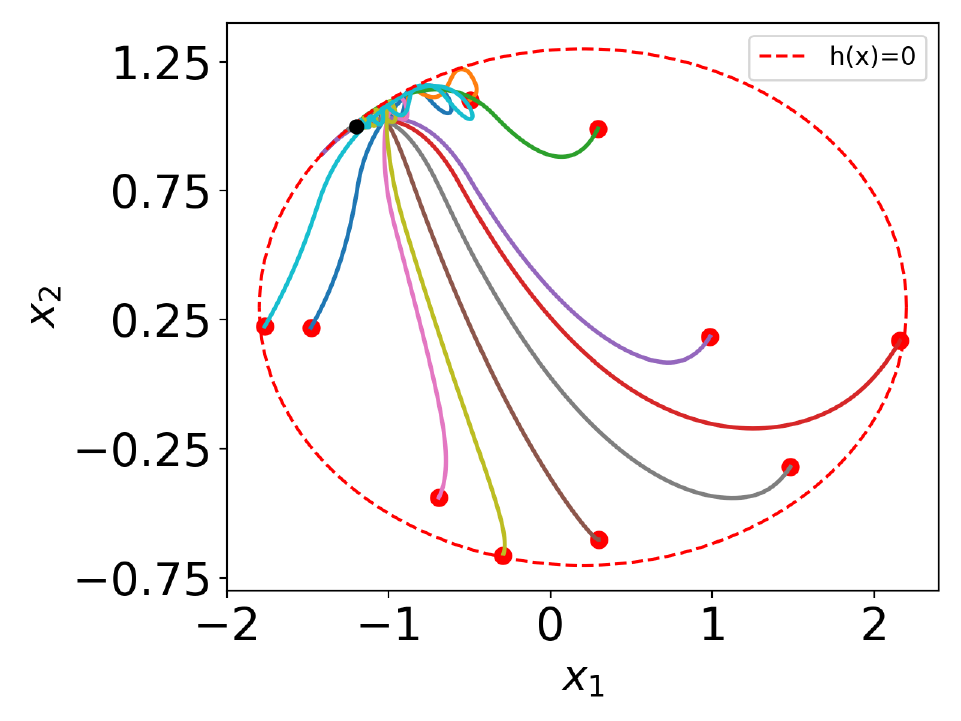}}
        \subfigure[]{\includegraphics[height=.24\linewidth]{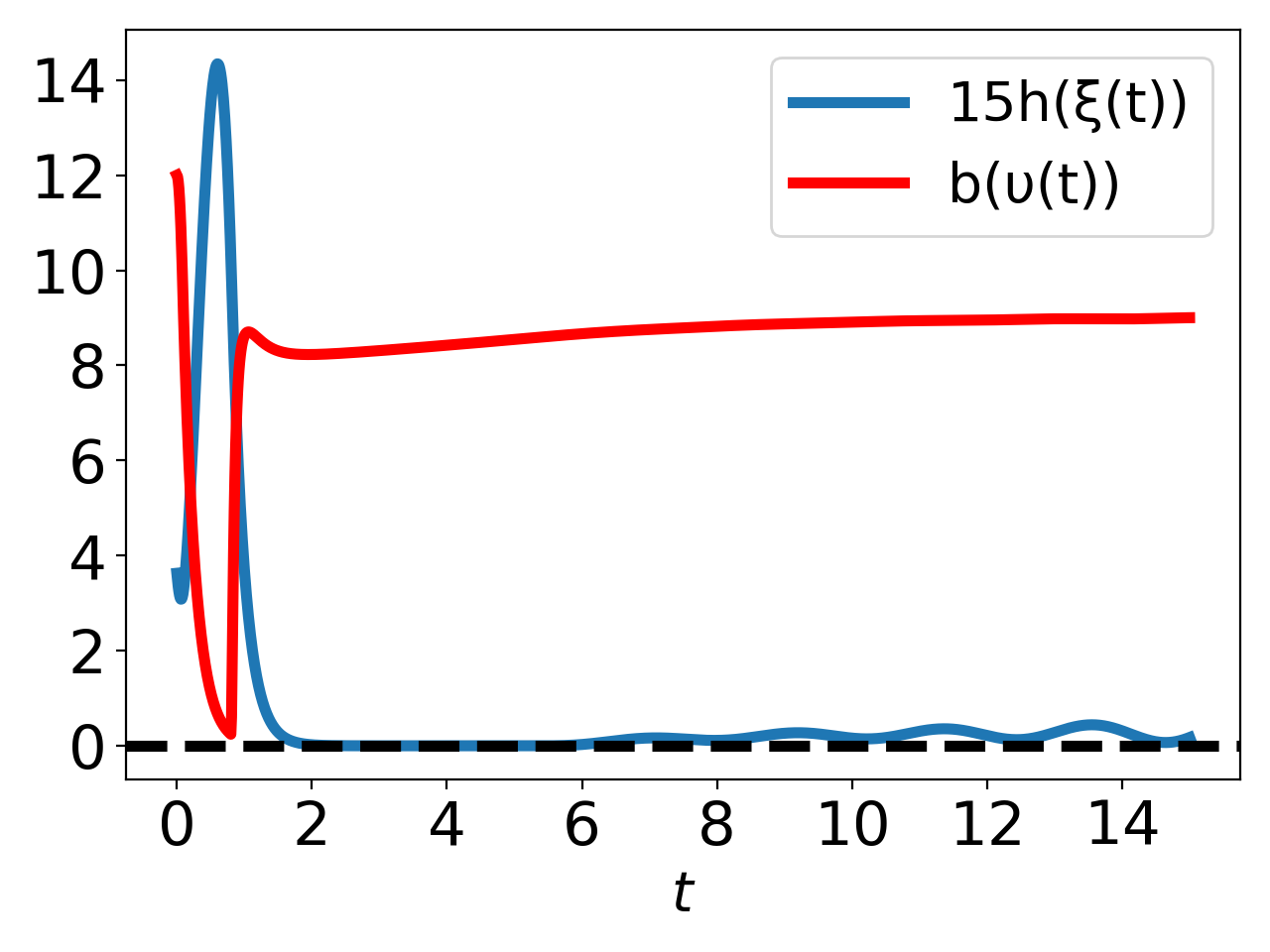}}
        \hspace*{-2ex}
        \subfigure[]{\includegraphics[height=.24\linewidth]{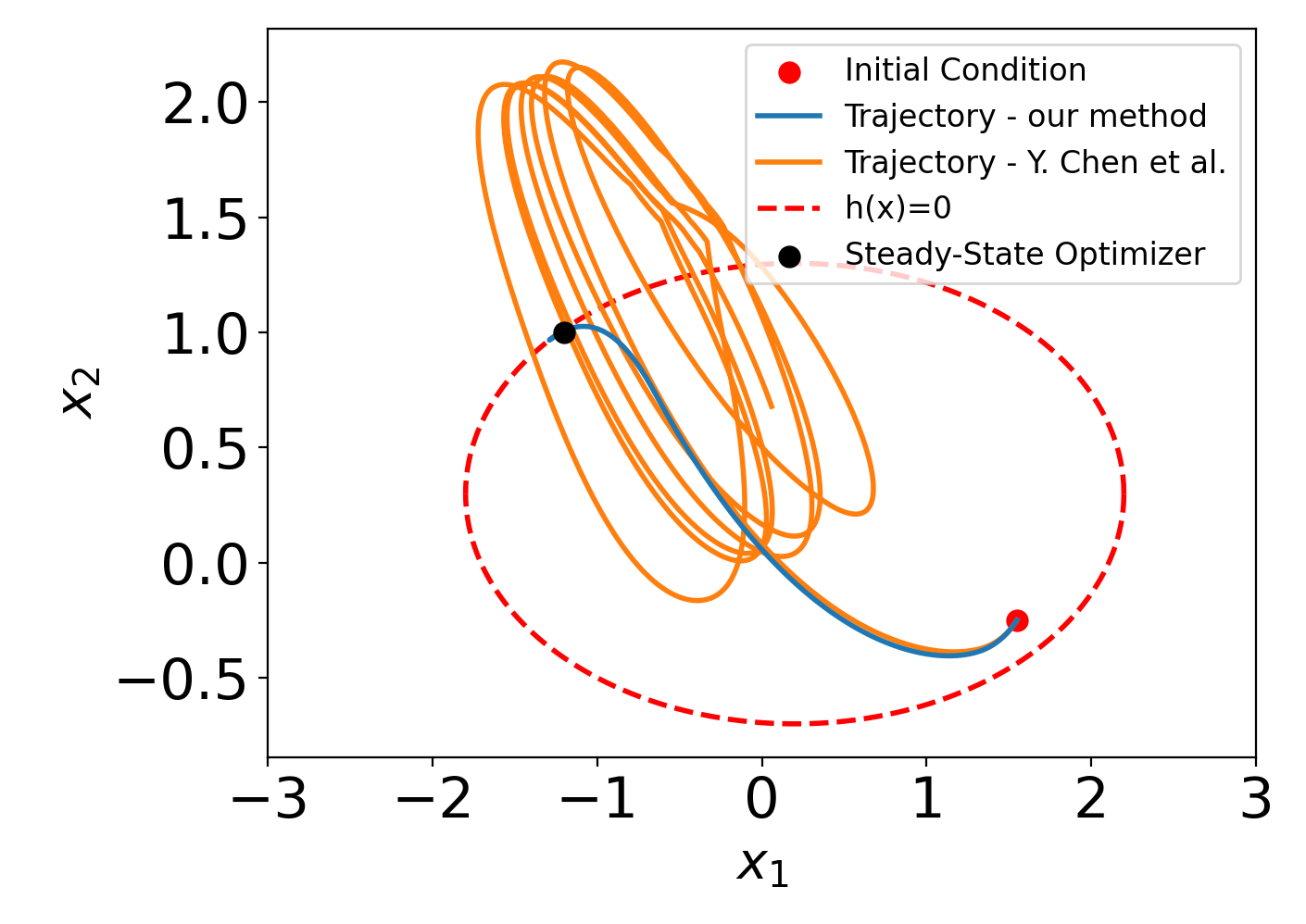}}
        \caption{Implementation of~\eqref{eq:closed_loop_sgf} in a convex feedback optimization problem (cf. Section~\ref{sec:convex-case}): (a) trajectories from different initial conditions, (b) evolution of $h$ and $b$ along the trajectory with initial condition at $x_0 = (1.55, -0.25), u_0 = (0, 0)$, and (c) comparison with the approach in~\cite{YC-LC-JC-EDA:23-csl}.}
        \label{fig:convex}
\end{figure*}

\subsection{Convex case}\label{sec:convex-case}
We consider the control system 
\begin{align*}
    \dot{\xi} = A\xi + B\upsilon,  \; A = \begin{pmatrix}
        -1.3 & -0.3 \\
        -1.4 & -0.5
    \end{pmatrix}, \; 
    B = \begin{pmatrix}
        1 & 0 \\
        0 & 1
    \end{pmatrix}.
\end{align*}
Note that it satisfies Assumption~\ref{assum:stability} and the steady-state map is $w(u) = -A^{-1}Bu$.
Consider the feedback optimization problem~\eqref{eq:steady-state_opti_problem} with objective and constraints given by
\begin{align*}
    \Phi(x) &= (x_1 + 1.2)^2 + (x_2 - 1)^2,
    \\
    h(x) &= 1 - \frac{(x_1-0.2)^2}{4} - \frac{(x_2-0.3)^2}{1}, \ b(u) = 12 - \norm{u}^2.
\end{align*}
In this case,~\eqref{eq:steady-state_opti_problem} is strictly convex and its unique minimizer is $(x_*, u_*) = (-1.2, 1.0, -1.26, -1.18)$, which satisfies $h(x_*)>0$. Figure~\ref{fig:convex} presents the result of the implementation of~\eqref{eq:closed_loop_sgf}. Figure~\ref{fig:convex}(a) shows the evolution of multiple trajectories, 
Figure~\ref{fig:convex}(b) shows the evolution of $h$ and $b$ along the trajectory with initial condition at $x_0 = (1.55, -0.25), u_0 = (0, 0)$, and Figure~\ref{fig:convex}(c) compares our approach with the one in~\cite{YC-LC-JC-EDA:23-csl}, which only guarantees anytime satisfaction of the input constraints (it enforces state constraints only asymptotically). Although both approaches converge to the steady-state optimizer, \cite{YC-LC-JC-EDA:23-csl}~violates of state constraints during transients.

\begin{figure*}[tbh]
     \centering
       \subfigure[]{\includegraphics[height=.24\linewidth]{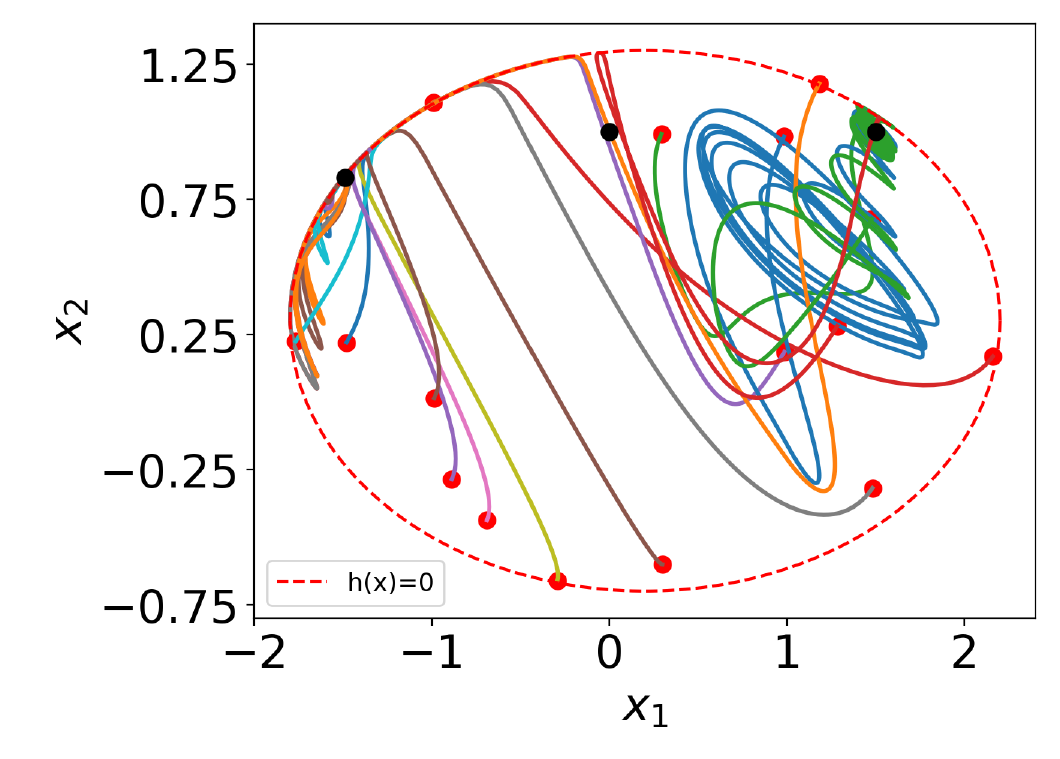}}
       \subfigure[]{\includegraphics[height=.24\linewidth]{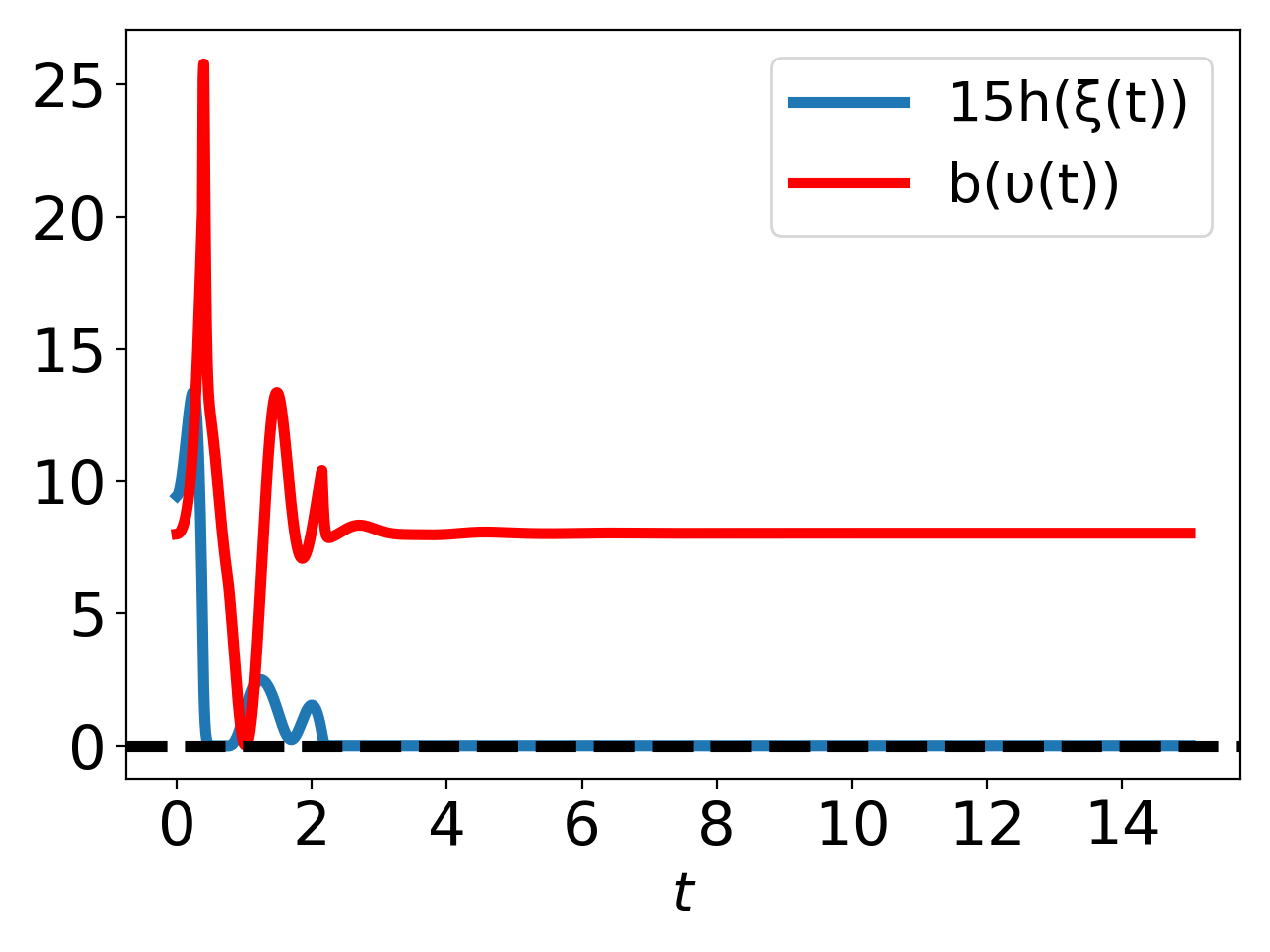}}
       \hspace*{-2ex}
       \subfigure[]{\includegraphics[height=.24\linewidth]{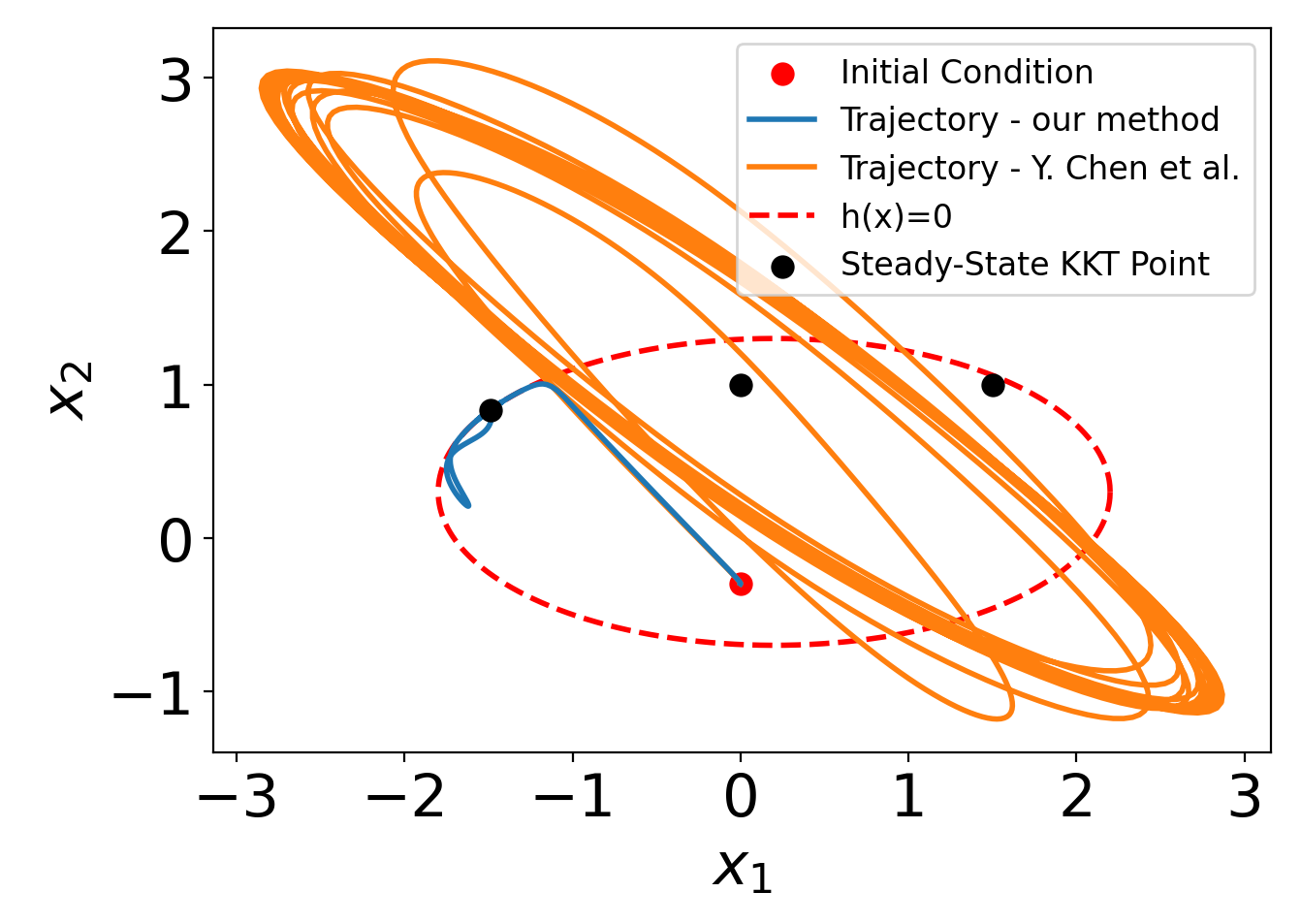}}    
     \caption{Implementation of~\eqref{eq:closed_loop_sgf} in a non-convex feedback optimization problem (cf. Section~\ref{sec:not-convex}): (a) trajectories from different initial conditions, (b) evolution of $h$ and $b$ along the trajectory with initial condition at $x_0 = (1.55, -0.25), u_0 = (0, 0)$, and (c) comparison with the approach in~\cite{YC-LC-JC-EDA:23-csl}.}\label{fig:non-convex}
\end{figure*}

\subsection{Non-convex case}\label{sec:not-convex}

We consider the same control system as in Section~\ref{sec:convex-case}.
Consider the feedback optimization problem~\eqref{eq:steady-state_opti_problem} with the same $h$ as in the convex case, but instead take
\begin{align*}
    \Phi(x) &= (x_1^2 - 1.5^2)^2 + (x_2^2-1)^2,
   \
    b(u) = 8 + u_1^2 - u_2^2.
    \end{align*}
In this case,~\eqref{eq:steady-state_opti_problem} is non-convex and has three KKT points at $(0,1), (1.5, 1.0)$ and $(-1.49, 0.8)$. 
Furthermore, $(1.5, 1.0)$ and $(-1.49, 0.8)$ are local optima, and hence, by Theorem~\ref{thm:local_convergence}, locally asymptotically stable under~\eqref{eq:closed_loop_sgf}.
This is illustrated in Figure~\ref{fig:non-convex}(a). Note that $(0,1)$ is not a local optimum, and hence it is not locally asymptotically stable.
Figure~\ref{fig:non-convex}(b) showcases the evolution of $h$ and $b$ along the trajectory with initial condition $x_0 = (0, -0.3)$, $u_0 = (0, 0)$ and Figure~\ref{fig:non-convex}(c) shows the different trajectories obtained by our approach and that of~\cite{YC-LC-JC-EDA:23-csl}, which leads again to state violations.

\subsection{Discussion on item \ref{it:weird-assumption} of Theorem~\ref{thm:global_convergence}}\label{sec:discussion-assumption-global-convergence}

Here we study the conservativeness of item~\ref{it:weird-assumption} in Theorem~\ref{thm:global_convergence} in a specific example.
Consider 
\begin{align*}
    &\dot{\xi} = A\xi + B\upsilon, \ A = \begin{pmatrix}
        -1 & 1 \\
        0 & -1
    \end{pmatrix}, \ B =\begin{pmatrix}
        0 \\ 1
    \end{pmatrix},\\
    &\Phi(x) = \norm{x-x_{\text{ss}}}^2, \ h(x) = 4 - \norm{x}^2, \ b(u) = 4 - u^2,
\end{align*}
where $x_{\text{ss}} \in \real^2$ is fixed.
Note that for any $u_* \in [-\sqrt{2}, \sqrt{2}]$, 
by taking $x_{\text{ss}} = (u_*, u_*)$, we get that the optimizer is $(x_{\text{ss}}, u_*)$.

It can be shown that 
item~\ref{it:weird-assumption} in Theorem~\ref{thm:global_convergence} holds only when the optimizer $(x_*,u_*)$ is such that $u_* \in [-1.2, 1.2]$.
Therefore, Theorem~\ref{thm:global_convergence} ensures that whenever $u_*\in[-1.2, 1.2]$, the unique optimizer of~\eqref{eq:steady-state_opti_problem} is globally asymptotically stable.
Nonetheless, we have checked numerically that the global asymptotic stability of the unique minimizer of~\eqref{eq:steady-state_opti_problem} also holds whenever $u_*\notin[-1.2,1.2]$ (see Figure~\ref{fig:trajectories-assumption-global-convergence}, where $u_*=1.4$), indicating that even if item~\ref{it:weird-assumption} in Theorem~\ref{thm:global_convergence} is not satisfied, global asymptotic stability can still hold. 

\begin{figure}[htb]
    \centering
    \includegraphics[width=0.85\linewidth]{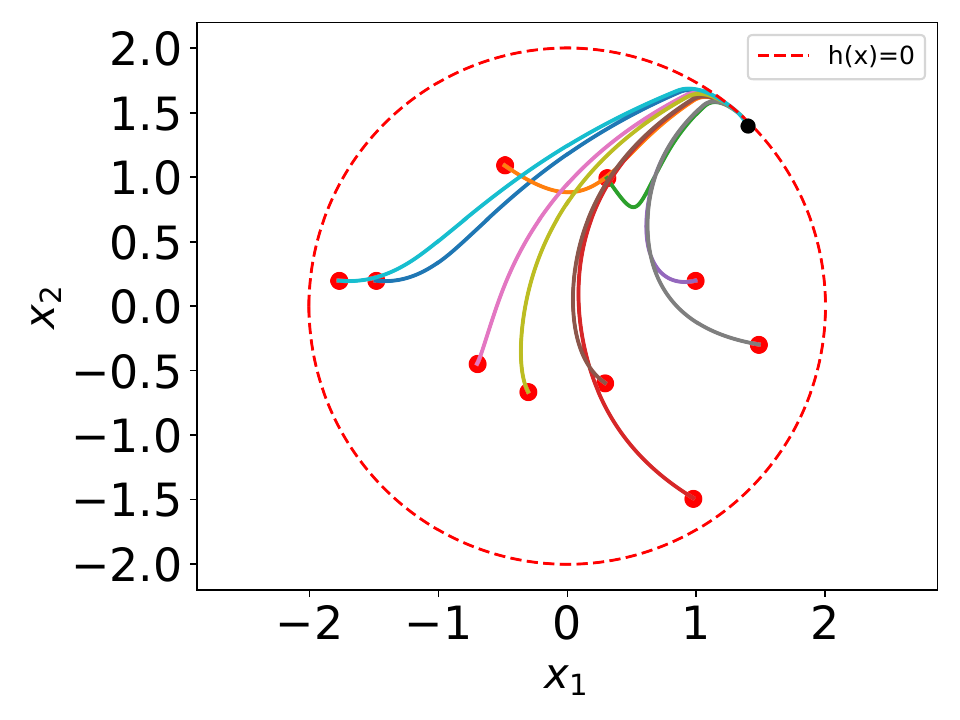}
    \includegraphics[width=0.45\linewidth]{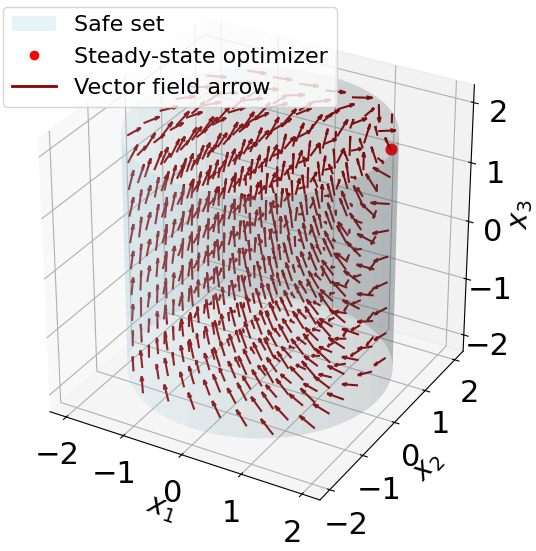}
    \includegraphics[width=0.45\linewidth]{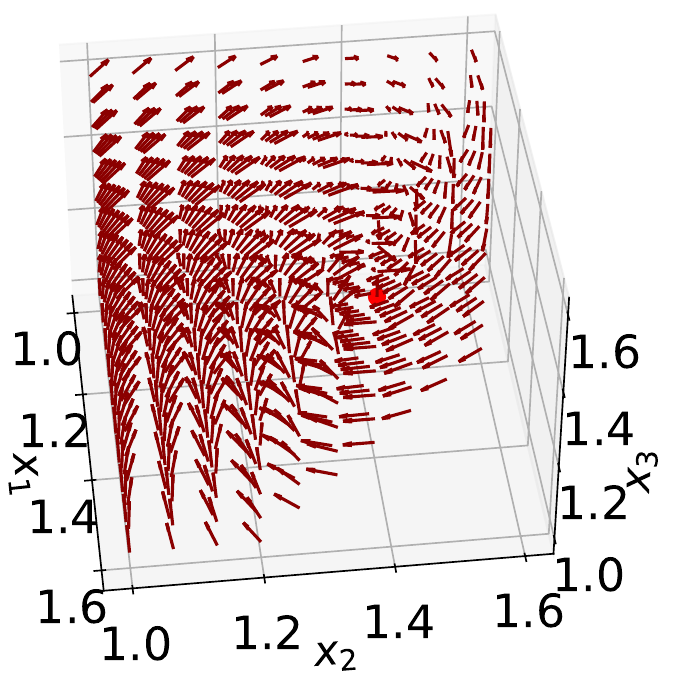}

    \caption{(top) Evolution of trajectories under~\eqref{eq:closed_loop_sgf} for the example in Section~\ref{sec:discussion-assumption-global-convergence}, showing global asymptotic stability of the unique minimizer even when item \ref{it:weird-assumption} of Theorem~\ref{thm:global_convergence} is not satisfied. (bottom left) Phase portrait of the corresponding closed-loop vector field. (bottom right) Zoomed in version of said phase portrait.}
    \label{fig:trajectories-assumption-global-convergence}
\end{figure}

\subsection{Undesirable equilibria, optimum at the boundary and regularization}\label{sec:undesirable-eq}

Here we present an example where the conditions in Proposition~\ref{prop:critical_points_equilibria}
are not satisfied and the closed-loop system~\eqref{eq:closed_loop_sgf} has equilibrium points which do not correspond to KKT points of the feedback optimization problem~\eqref{eq:steady-state_opti_problem}. Such spurious equilibria often occur in the implementation of control barrier function-based controllers~\cite{tan2021undesired,reis2020control,PM-YC-EDA-JC:25-jnls}. Consider 
\begin{align*}
    &\dot{\xi} = A\xi + B\upsilon, \ 
     A = \begin{pmatrix}
        -2 & 0 \\
        0 & -4
    \end{pmatrix}, \quad B = \begin{pmatrix}
        1 & 0 \\
        0 & 1
    \end{pmatrix},\\
    &h(x) = 1 - x_1 - x_2, \; 
    \Phi(x) = \norm{x-x_{\text{ss}}}^2, \quad x_{\text{ss}} = (2,2) .
\end{align*}
It can be shown that the optimizer of~\eqref{eq:steady-state_opti_problem} is at $(0.5, 0.5)$, whereas the closed-loop's (undesirable) equilibrium is at $(1,0)$. Figure~\ref{fig:trajectories-spurious-equilibrium} shows that $(1,0)$ is asymptotically stable.
\begin{figure}[htb]
    \centering
    \includegraphics[width=0.9\linewidth]{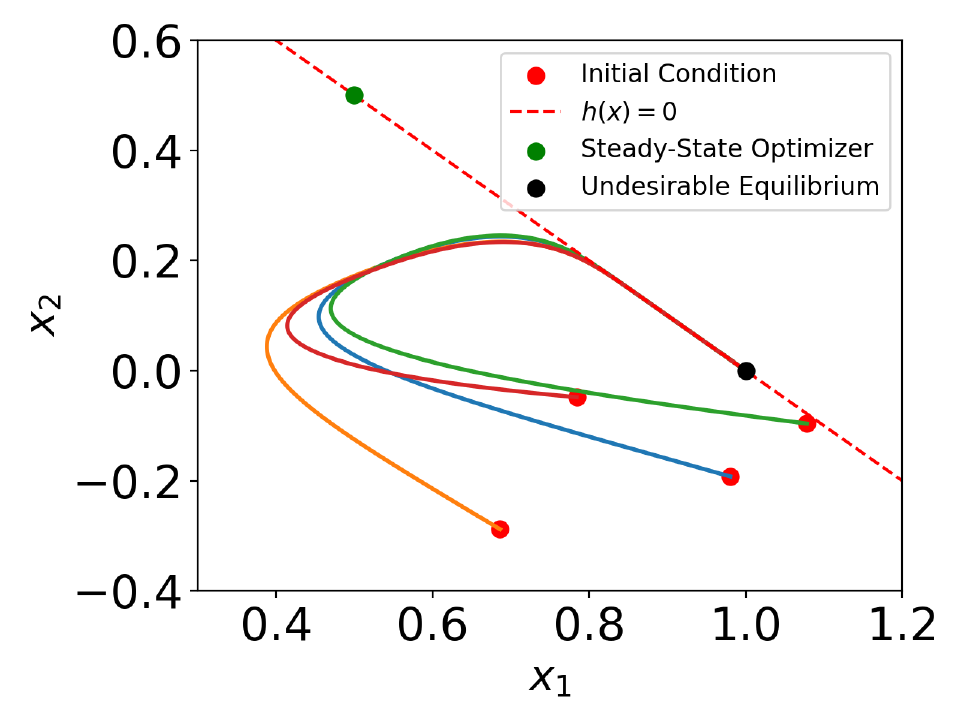}
    \caption{Evolution of various trajectories under~\eqref{eq:closed_loop_sgf} for the example in Section~\ref{sec:undesirable-eq}. The point $(1,0)$ is a locally asymptotically stable equilibrium which does not correspond to the steady state-optimizer of~\eqref{eq:steady-state_opti_problem}.}
    \label{fig:trajectories-spurious-equilibrium}
\end{figure}
Figure~\ref{fig:trajectories-spurious-regularized} shows that by adding the regularization term from Proposition~\ref{prop:regularization} (with $p = 2$, $\epsilon = 0.8$), we obtain a slightly suboptimal asymptotically stable optimizer at $(0.48, 0.48)$.
\begin{figure}[htb]
    \centering
    \includegraphics[width=0.9\linewidth]{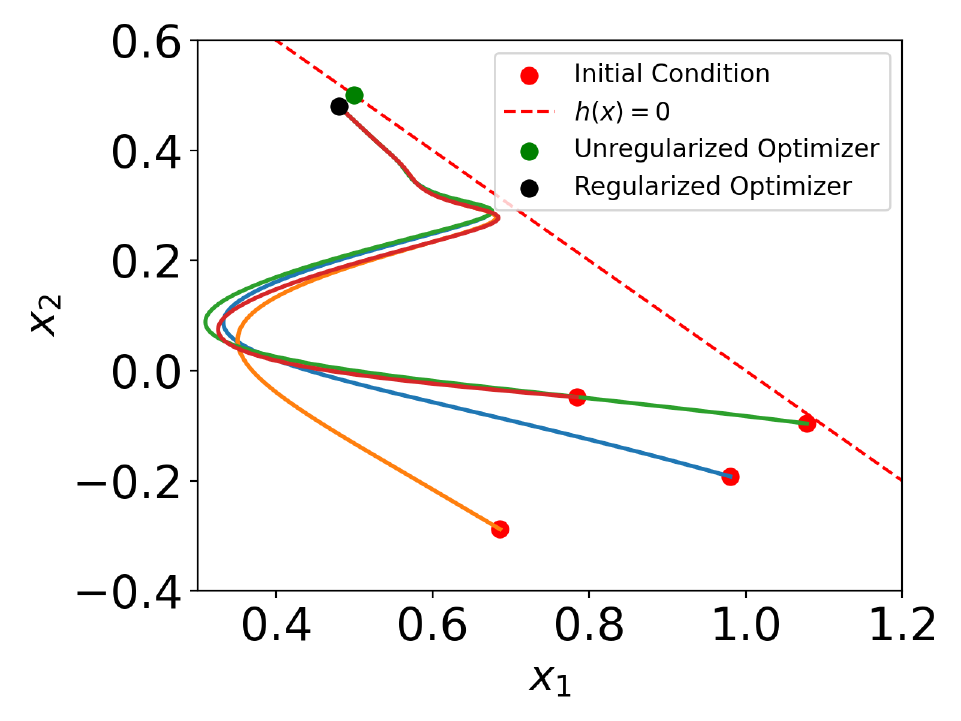}
    \caption{Evolution of various trajectories under~\eqref{eq:closed_loop_sgf} for the example in Section~\ref{sec:undesirable-eq} with the regularization term in Proposition~\ref{prop:regularization}. The point $(0.48,0.48)$ is an asymptotically stable equilibrium and a steady state-optimizer of the regularized version of~\eqref{eq:steady-state_opti_problem}.}
    \label{fig:trajectories-spurious-regularized}
\end{figure}

\section{Conclusions}
We have presented a novel feedback-optimization method that enforces state constraints at all times by employing high-order CBFs and the safe gradient flow. We have characterized the properties of the proposed controller regarding well-posedness, safety, equivalence between equilibria and critical points of the optimization problem, and local and global asymptotic stability of optima. Our approach makes use of the system dynamics to enforce safety. Future work will focus on scenarios where the dynamics are (partially) unknown. Towards this direction, we will explore the use of results in the CBF literature related to \emph{input-to-state safety}, such as \cite{kolathaya2018input}. Finally, one may also employ results as in \cite{PM-AM-JC:25-l4dc}, combining safe-gradient flows with stochastic approximation, to extend the treatment to unknown objective functions. 

\section{Acknowledgements} The authors would like to thank Riccardo Bertollo, for his insights in proving Lemma \ref{lem:ric_bound}.

\section{Proofs and auxiliary technical results}\label{sec:proofs}

{\color{black}

\begin{proof}[\textbf{Proof of Proposition~\ref{prop:feasibility}}]
    For $(x,u)\in\partial\Cc_r$, let $\Nc_{(x,u)}$ be an open neighborhood of $(x,u)$ such that, for all $(\bar{x},\bar{u})\in\Nc_{(x,u)}$,
    \begin{itemize}
        \item if $b(u)=0$, then $\nabla^\top b(\bar{u}) q_{(x,u)} > 0$;
        \item if $h_r(x,u)=0$, then $\frac{\partial h_r}{\partial u}(\bar{x},\bar{u})q_{(x,u)} + \frac{\partial h_r}{\partial x}(\bar{x},\bar{u})f(\bar{x},\bar{u})>0$.
    \end{itemize}
    Note that such a neighborhood $\Nc_{(x,u)}$ exists by continuity of $\nabla b$, $\frac{\partial h_r}{\partial u}$ and $(x,u)\mapsto \frac{\partial h_r}{\partial x}(x,u)f(x,u)$.
     
    Now, let $S_1 := \setdef{(x,u)\in\Cc_r}{b(u) = 0}$, $S_2 := \setdef{(x,u)\in\Cc_r}{h_r(x,u) = 0}$, $S_3 :=S_1\cap S_2$, and
    consider the following three sets:
    \begin{align*}
        &\Nc_1 = \hspace{-2mm}\bigcup_{ (x,u)\in S_1 } \hspace{-2mm}\Nc_{(x,u)}, \ \Nc_2 = \hspace{-2mm}\bigcup_{ (x,u)\in S_2 }\hspace{-2mm} \Nc_{(x,u)}, \ \Nc_3 = \hspace{-2mm}\bigcup_{ (x,u)\in S_3 } \hspace{-2mm}\Nc_{(x,u)}.
    \end{align*}
    Note that $S_1, S_2, S_3$ are compact because $\Cc_r$ is compact and the functions $b$ and $h_r$ are continuous.
    Since $\Nc_1$ is a cover of $S_1$, we can extract a finite sub-cover $\Nc_{f,1} = \bigcup_{i=1}^{N_1} \Nc_{(x_i,u_i)}$ of $S_1$ with $b(u_i) = 0$ for all $i\in [N_1]$ and $N_1 \in \mathbb{Z}_{>0}$.
    Similarly, we can extract finite sub-covers 
    $\Nc_{f,2} = \bigcup_{i=1}^{N_2} \Nc_{(x_i,u_i)}$ of $S_2$, with $h_r(x_i,u_i) = 0$ for all $i\in[N_2]$ and $N_2 \in \mathbb{Z}_{>0}$, and 
    $\Nc_{f,3} = \bigcup_{i=1}^{N_3} \Nc_{(x_i,u_i)}$ of $S_3$, with $b(u_i) = h_r(x_i,u_i) = 0$ for all $i\in[N_3]$ and $N_3 \in \mathbb{Z}_{>0}$.

    Since $b$ is continuous and $S_1$ is compact, there exists $c_1>0$ such that,
    for any $(x,u)\in\Cc_r$ with
    $b(u) \in [0, c_1]$, we have $(x,u)\in\Nc_{f,1}$. Indeed, if this was not the case, there would exist a sequence of positive numbers $\{ \bar{c}_i \}_{i\in\mathbb{Z}_{>0}}$ with $\lim\limits_{i\to\infty} \bar{c}_i = 0$ and points $(\bar{x}_i,\bar{u}_i)\in \Cc_r$ such that $b(\bar{u}_i) \in [0, \bar{c}_i]$ and $(\bar{x}_i,\bar{u}_i)\notin\Nc_{f,1}$. As $\Cc_r\setminus\Nc_{f,1}$ is compact, a subsequence of these points converges to a point $(x_*,u_*)\in \Cc_r\setminus\Nc_{f,1}$ with $b(u_*)=0$; this is a contradiction, as $b(u_*)=0$ implies $(x_*,u_*)\in S_1\subseteq \Nc_{f,1}$. Similarly, there exists $c_2>0$ such that for any $(x,u)\in\Cc_r$ with $h_r(x,u)\in[0,c_2]$, then $(x,u)\in\Nc_{f,2}$; and $c_3>0$ such that, for any $(x,u)\in\Cc_r$ with $b(u) \in [0,c_3]$ and $h_r(x,u)\in[0,c_3]$, then $(x,u)\in\Nc_{f,3}$.

    Next, we consider four cases:
    \newline
  \mbox{} \quad 1) Let $Q_1 = \setdef{(x,u)\in \Cc_r}{b(u)\in[0,c_1], h_r(x,u)>c_3 }$. Since $Q_1 \subseteq \Nc_{f,1}$, for each $(x,u)\in Q_1$, we know that there exists $i\in[N_1]$ such that $\nabla b(u)^\top q_{(x_i,u_i)} + \alpha b(u) > 0$ for all $\alpha\geq 0$. Define
        \begin{align*}
            \gamma_{f,1} \! = \! \frac{1}{c_3} \max\limits_{ 
            \substack{
            (x,u)\in \Cc_r, \\ i\in[N_1]
            }
            } 
            \!
            \Big| &\frac{\partial h_r}{\partial u}(x,u) q_{(x_i,u_i)}
            \! + \! \frac{\partial h_r}{\partial x}(x,u)f(x,u) \Big|.
        \end{align*}
        Since $\Cc_r$ is compact, $\frac{\partial h_r}{\partial u}$ and $(x,u)\mapsto\frac{\partial h_r}{\partial x}(x,u)f(x,u)$ are continuous, and $N_1$ is finite, $\gamma_{f,1}$ is finite. Therefore, $\frac{\partial h_r}{\partial u}(x,u) q_{(x_i,u_i)} + \frac{\partial h_r}{\partial x}(x,u)f(x,u) + \gamma h_r(x,u) > 0$, for all $\gamma > \gamma_{f,1}$. Thus, QP \eqref{eq:sgf_QP} is feasible in $Q_1$, for all $\alpha\geq0$, $\gamma > \gamma_{f,1}$.
        \newline
        \mbox{} \quad 2) Let $Q_2 = \setdef{(x,u)\in\Cc_r}{ h_r(x,u)\in[0,c_2], b(u)>c_3 }$. Since $Q_2 \subseteq \Nc_{f,2}$, for each $(x,u)\in Q_2$, we know that there exists $i\in[N_2]$ such that
        \begin{equation*}
            \frac{\partial h_r}{\partial u}(x,u) q_{(x_i,u_i)} + \frac{\partial h_r}{\partial x}(x,u)f(x,u) + \gamma h_r(x,u) > 0,
        \end{equation*}
        for any $\gamma\geq0$. Define 
        \begin{align*}
            \alpha_{f,1} := \frac{ \max\limits_{ (x,u)\in \Cc_r, i\in[N_2] } |\nabla b(u)^\top q_{(x_i,u_i)}| }{c_3}.
        \end{align*}
        Since $\Cc_r$ is compact, $\nabla b$ is continuous, and $N_2$ is finite, $\alpha_{f,1}$ is finite. Therefore, $\nabla b(u)^\top q_{(x_i,u_i)} + \alpha b(u) \geq 0$ for any $\alpha > \alpha_{f,1}$. Thus, QP \eqref{eq:sgf_QP} is feasible in $Q_2$, for all $\alpha > \alpha_{f,1}$, $\gamma\geq0$.
        \newline
        \mbox{} \quad 3) Let $Q_3 = \setdef{(x,u)\in\Cc_r}{h_r(x,u)\in[0,c_3], b(u)\in[0,c_3]}$. Since $Q_3 \subseteq \Nc_{f,3}$, 
        then, for each $(x,u)\in Q_3$,
        there exists $i\in[N_3]$ such that 
        \begin{align*}
            &\nabla b(u)^\top q_{(x_i,u_i)} + \alpha b(u) > 0, \\
            &\frac{\partial h_r}{\partial u}(x,u) q_{(x_i,u_i)} + \frac{\partial h_r}{\partial x}(x,u)f(x,u) + \gamma h_r(x,u) > 0,
        \end{align*}
        for all $\alpha\geq 0$, $\gamma\geq0$. Thus, QP \eqref{eq:sgf_QP} is feasible in $Q_3$, for all $\alpha\geq 0$, $\gamma\geq0$.
\newline
        \mbox{} \quad 4) Let $Q_4 = \Cc_r \setminus ( Q_1 \cup Q_2 \cup Q_3 )$.
        Note that for all $(x,u) \in Q_4$, we have 
        $b(u) > \min\{ c_1, c_3 \} > 0$ and 
        $h_r(x,u) > \min\{ c_2, c_3 \} > 0$. Define 
        \begin{align*}
            \gamma_{f,2} = \frac{ \max\limits_{ (x,u)\in (\Sc\times\mathcal{U})\cap\Cc_r } | \frac{\partial h_r}{\partial x}(x,u)f(x,u) | }{ \min\{ c_2, c_3 \} }.
        \end{align*}
        Since $\Cc_r$ is compact, and $\frac{\partial h_r}{\partial x}$ and $f$ are continuous, $\gamma_{f,2}$ is finite. Now, for all $(x,u)\in Q_4$ we have 
        \begin{align*}
            &\nabla b(u)^\top \mathbf{0}_m + \alpha b(u) > 0, \\
            &\frac{\partial h_r}{\partial u}(x,u)^\top \mathbf{0}_m + \frac{\partial h_r}{\partial x}(x,u)f(x,u) + \gamma h_r(x,u) > 0.
        \end{align*}
        for all $\alpha \geq 0$ and $\gamma \geq \gamma_{f,2}$. Thus, QP \eqref{eq:sgf_QP} is feasible in $Q_4$, for all $\alpha > 0$, $\gamma > \gamma_{f,2}$.
   
    The result follows by observing that $\Cc_r = Q_1\cup Q_2 \cup Q_3 \cup Q_4$, and taking $\alpha_f > \alpha_{f,1}$, $\gamma_f > \max\{ \gamma_{f,1}, \gamma_{f,2} \}$.
\end{proof}

\begin{proof}[\textbf{Proof of Proposition \ref{prop:solutions_sgf}}]
    Note that the assumptions of Proposition~\ref{prop:solutions_sgf} include those of Proposition~\ref{prop:feasibility}.
    Hence, by following the argument used in Proposition~\ref{prop:feasibility}, by taking $\alpha > \alpha_f$, $\gamma > \gamma_f$,~\eqref{eq:sgf_QP} is strictly feasible (i.e., for any $(x,u)\in\Cc_r$, there exists $q\in\real^m$ satisfying each of the constraints strictly).
    Thus, Slater's condition holds for~\eqref{eq:sgf_QP}. Since~\eqref{eq:sgf_QP} is a strongly convex QP, by~\cite[Proposition 5.39]{NA-AE-MP:20}, we have that MFCQ holds for~\eqref{eq:sgf_QP} at the optimizer. 
    Furthermore, 
    since CRCQ holds at every $(x,u)\in\Cc_r$ for~\eqref{eq:sgf_QP},
    the result follows from~\cite[Theorem 3.6]{JL:95}.
\end{proof}

\begin{proof}[\textbf{Proof of Proposition \ref{prop:crcq}}]
    Since $b$ and $h_r$ are continuously differentiable, $\mathcal{C}_r$ is compact, 
    $b(u) = 0$ implies $\nabla b(u) \neq 0$ (by item~\ref{it:diff-regularity-second} in Assumption~\ref{assum:differentiability, regularity and relative degree})
    and $h_r(x,u) = 0$ implies $\frac{\partial h_r}{\partial u}(x,u) \neq 0$ (by item~\ref{it:diff-regularity-first} in Assumption~\ref{assum:differentiability, regularity and relative degree}),
    there exist $\delta_1, \delta_2, \delta_3 >0$ sufficiently small such that 
    \begin{itemize}
        \item $b(u) \in [0,\delta_1]$ implies $\nabla b(u) \neq 0$;
        \item $h_r(x,u) \in [0,\delta_2]$ implies $\frac{\partial h_r}{\partial u}(x,u) \neq 0$;
        \item $b(u) \in [0,\delta_3]$ and $h_r(x,u) \in [0,\delta_3]$ implies that $\nabla b(u)$, $\frac{\partial h_r}{\partial u}(x,u)$ are linearly independent.
    \end{itemize}

    To prove the last point, suppose that it does not hold. Then, there exists a sequence of points $\{ (x_i, u_i) \}_{i\in\mathbb{Z}_{>0}}$
    and a sequence of strictly positive numbers $\{ \delta_i \}_{i\in\mathbb{Z}_{>0}}$,
    such that $b(u_i)\in[0,\delta_i]$, $h_r(x_i,u_i)\in[0,\delta_i]$, and $\nabla b(u_i)$, $\frac{\partial h_r}{\partial u}(x_i,u_i)$ are linearly dependent. Then, we have 
    $|\nabla b(u_i)^\top \frac{\partial h_r}{\partial u}(x_i,u_i)| - \norm{ \nabla b(u_i) } \norm{\frac{\partial h_r}{\partial u}(x_i,u_i)} = 0$ for all $i\in\mathbb{Z}_{>0}$.
    Since $\mathcal{C}_r$ is compact, there exists a convergent subsequence $( x_{i_k}, u_{i_k} )_{k\in\mathbb{Z}_{>0}}$.
    Note that necessarily, $(x_{\infty},u_{\infty}):=\lim\limits_{k\to\infty} (x_{i_k}, u_{i_k})$ satisfies $b(u_{\infty}) = h_r(x_{\infty},u_{\infty}) = 0$.
    Moreover, by continuity of $\nabla b$ and $\frac{\partial h_r}{\partial u}$, we have 
    \begin{align*}
        |\nabla b(u_{\infty})^\top \frac{\partial h_r}{\partial u}(x_{\infty},u_{\infty}) | - \| \nabla b(u_{\infty}) \| \Big\| \frac{\partial h_r}{\partial u}(x_{\infty},u_{\infty}) \Big\| =& 
        \\
        \lim\limits_{k\to\infty} |\nabla b(u_{i_k})^\top \frac{\partial h_r}{\partial u}(x_{i_k},u_{i_k})| - \norm{ \nabla b(u_{i_k}) } \Big\| \frac{\partial h_r}{\partial u}(x_{i_k},u_{i_k}) \Big\| \ & \\
        = 0,\ &
    \end{align*}
    contradicting the assumption that if $b(u) = 0$ and $h_r(x,u) = 0$, then $\nabla b(u)$ and $\frac{\partial h_r}{\partial u}(x,u)$ are linearly independent.

    Let $\delta = \min \{ \delta_1, \delta_2, \delta_3 \}$.
    Now, define 
    \begin{align*}
        g_b(x,u) &= \argmin\limits_{q} \frac{1}{2}\norm{q + \epsilon \frac{\partial w}{\partial u}(u)^\top \nabla\Phi(x) }^2 \\
        &\quad \text{s.t.} \ \nabla b(u)^\top q + \alpha b(u) \geq 0.
    \end{align*}
    Note that $q = 0$ is feasible for the optimization problem defining $g_b$ for all $(x,u)\in\mathcal{C}_r$.
    Therefore,
    \begin{align*}
        &\norm{g_b(x,u)} - \norm{\epsilon \frac{\partial w}{\partial u}(u)^\top \nabla\Phi(x)} \leq \\
        &\norm{g_b(x,u) + \epsilon \frac{\partial w}{\partial u}(u)^\top \nabla\Phi(x)} \leq \norm{0 + \epsilon \frac{\partial w}{\partial u}(u)^\top \nabla\Phi(x)},
    \end{align*}
    which implies that $\norm{g_b(x,u)} \leq 2\epsilon \norm{\frac{\partial w}{\partial u}(u)^\top \nabla\Phi(x)}$.
    Since $\frac{\partial w}{\partial u}$ and $\nabla \Phi$ are continuous, there exists $M_b>0$ (independent of $\alpha$) such that $\norm{g_b(x,u)} \leq M_b$ for all $(x,u)\in\mathcal{C}_r$.
    Now, let
    \begin{align*}
        \gamma > \frac{ \max\limits_{(x,u)\in\mathcal{C}_r } |\frac{\partial h_r}{\partial u}(x,u)^\top g_b(x,u) + \frac{\partial h_r}{\partial x}(x,u)^\top f(x,u)|} {\delta} =: \gamma_1.
    \end{align*}
    Note that $\gamma_1$ is finite because $\norm{g_b(x,u)} \leq M_b$ for all $(x,u)\in\mathcal{C}_r$, $\mathcal{C}_r$ is compact, and $\frac{\partial h_r}{\partial u}$, $\frac{\partial h_r}{\partial x}$, and $f$ are continuous.
    
    With this selection of $\gamma$, 
    the CBF constraint associated with $h_r$ in~\eqref{eq:sgf_QP} is not active when $h_r(x,u) \geq \delta$.
    Indeed, note that since~\eqref{eq:sgf_QP} has the same objective function as the optimization problem defining $g_b$ and one additional constraint, we necessarily have that: 
    \begin{align*}
        \| g_{b}( x, \! u) \! + \! \epsilon \! \frac{\partial w}{\partial u}(u)^\top \! \nabla\Phi(x) \| \leq \!
        \| g_{\epsilon,\alpha,\gamma}(x, \! u) \! + \! \epsilon \! \frac{\partial w}{\partial u}(u)^\top \! \nabla\Phi(x) \|.
    \end{align*}
    However, since our selection of $\gamma$ ensures that $g_b$ is feasible for~\eqref{eq:sgf_QP} whenever $h_r(x,u)\geq\delta$,
    we necessarily have that $g_{\epsilon}(x,u) = g_b(x,u)$ and the constraint associated with $h_r$ is not active at the optimizer
    $g_{\epsilon}(x,u)$. Similarly, let 
    \begin{align*}
        &g_h(x,u) = \argmin\limits_{q} \frac{1}{2}\norm{q + \frac{\partial w}{\partial u}(u)^\top \nabla\Phi(x) }^2 
        \\
        &\quad \quad \text{s.t.} \ \frac{\partial h_r}{\partial u}(x,u)^\top q + \frac{\partial h_r}{\partial x}(x,u)f(x,u) + \gamma h_r(x,u) \geq 0.
    \end{align*}
    If $h_r(x,u) \in [0,\delta]$,
    \begin{align*}
        q = \frac{ 2 | \frac{\partial h_r}{\partial x}(x,u) f(x,u) | }{ \norm{\frac{\partial h_r}{\partial u}(x,u)}^2 }\frac{\partial h_r}{\partial u}(x,u)
    \end{align*}
    is feasible for the optimization problem defining $g_h$. On the other hand, by taking 
    \begin{align}\label{eq:beta-condition-2}
        \gamma > \frac{ \max\limits_{(x,u)\in\mathcal{C}_r} |\frac{\partial h_r}{\partial x}(x,u) f(x,u)| }{\delta} =: \gamma_2,
    \end{align}
    $q = 0$ is feasible for the optimization problem defining $g_h$, when $h_r(x,u)>\delta$.
    Therefore, by an analogous argument as in the case of $g_b$, there exists $M_h$ (independent of $\gamma$ provided that $\gamma$ satisfies~\eqref{eq:beta-condition-2}) such that $\norm{g_h(x,u)} \leq M_h$ for all $(x,u)\in\mathcal{C}_r$. Hence, by letting 
    \begin{align*}
        \alpha > \frac{ \max\limits_{(x,u)\in\mathcal{C}_r} |\nabla b(u)^\top g_h(x,u)| }{\delta} =: \alpha_0,
    \end{align*}
    we get that the CBF constraint associated with $b$ in~\eqref{eq:sgf_QP} is not active when $b(u) \geq \delta$.
    Let $\gamma_0 := \max\{ \gamma_1, \gamma_2 \}$.
    Suppose that $\alpha > \alpha_0$, $\gamma > \gamma_0$.

    We now have the following cases:
    \begin{itemize}
        \item if $b(u) > \delta$, $h_r(x,u) > \delta$, then the constraints of~\eqref{eq:sgf_QP} are not active at the optimizer and CRCQ holds;
        \item if $b(u) \leq \delta$, $h_r(x,u) > \delta$, only the CBF constraint associated with $b$ in~\eqref{eq:sgf_QP} can be active. If it is, CRCQ holds because $\nabla b(u) \neq 0$;
        \item if $b(u) > \delta$, $h_r(x,u) \leq \delta$, only the CBF constraint associated with $h_r$ in~\eqref{eq:sgf_QP} can be active at the optimizer.
        If it is, CRCQ holds, because $\frac{\partial h_r}{\partial u}(x,u) \neq 0$;
        \item if $b(u) \leq \delta$, $h_r(x,u) \leq \delta$, both constraints in~\eqref{eq:sgf_QP} can be active at the optimizer. Since $\nabla b$ and $\frac{\partial h_r}{\partial u}$ are linearly independent, $\nabla b(u) \neq 0$ and $\frac{\partial h_r}{\partial u} \neq 0$, CRCQ holds.
    \end{itemize}
    Therefore, CRCQ holds for all $(x,u)\in\mathcal{C}_r$.
\end{proof}

\begin{proof}[\textbf{Proof of Proposition \ref{prop:safety}}]
    Through standard CBF arguments (see e.g.~\cite{xiao2021high}), as solutions are unique, we obtain that $\Sc_r$ is forward invariant. Now, if we have $h_{r-1}(\xi(0),\upsilon(0))\geq 0$ and $h_r(\xi(t),\upsilon(t))\geq 0$ for all $t\geq 0$, again through standard CBF arguments, we have that $h_{r-1}(\xi(t),\upsilon(t))\geq 0$ for all $t\geq 0$, due to the definition of $h_r$. Thus, the set $\{(x,u):h_r(x,u)\geq 0, \ h_{r-1}(x,u)\geq 0\}$ is forward invariant. Using this argument recursively, we conclude that $\bigcap\limits_{i=0}^r\Sc_i$ is forward invariant.
\end{proof}

\begin{proof}[\textbf{Proof of Proposition \ref{prop:feasible-points-lie-in-bigcap}}]
    Since $(x_*,u_*)$ is feasible for~\eqref{eq:steady-state_opti_problem}, $h(x_*) \geq 0$ and $b(u_*) \ge 0$. Therefore, $(x_*,u_*)\in\Sc_0$.
    Now, note that $h_1(x_*,u_*) = \nabla h(x_*,u_*)^T f(x_*,u_*) + \beta h(x_*)$, and since $f(x_*,u_*)=f(w(u_*),u_*)=0$ and $\beta > 0$, we obtain $h_1(x_*,u_*) \geq 0$.
    In fact, for all $i\in[r]$, $h_i(x_*,u_*) = \beta h_{i-1}(x_*,u_*)$, implying that $h_i(x_*,u_*) \geq 0$ for all $i\in[r]$ and hence $(x_*,u_*)\in\bigcap\limits_{i=0}^r \Sc_i$.
\end{proof}

\begin{proof}[\textbf{Proof of Proposition \ref{prop:critical_points_equilibria}}]
    \textbf{Proof of $\implies$.} 
    Let $(x_*,u_*)$ be a critical point of~\eqref{eq:steady-state_opti_problem}.
    First, because $x_*=w(u_*)$, we have that $f(x_*,u_*)=f(w(u_*),u_*)=0$, from Assumption \ref{assum:differentiability, regularity and relative degree}. Since $(x_*,u_*)$ is a KKT point of~\eqref{eq:steady-state_opti_problem}, there exist $\lambda_h,\lambda_b\geq 0$ and $\mu\in\real^n$ such that the KKT conditions for \eqref{eq:steady-state_opti_problem} are satisfied
    \begin{align}\label{eq:equiibria_proof_kkt_feedback_opti}
        &\begin{bmatrix}
            \nabla \Phi(x_*) \\ 0 
        \end{bmatrix}  - \lambda_h \begin{bmatrix}
            \nabla h(x_*) \\ 0 
        \end{bmatrix} - \lambda_b \begin{bmatrix}
            0 \\ \nabla b(u_*) 
        \end{bmatrix} + \notag
        \\ 
        &\qquad\qquad\qquad\qquad\qquad \begin{bmatrix}
            \mu\\ -\dfrac{\partial w}{\partial u}(u_*)^\top\mu
        \end{bmatrix} = 0,\\
        &0\leq \lambda_h \perp h(x_*)\geq0, \quad 0\leq \lambda_b \perp b(u_*)\geq0, \quad x_* = w(u_*) \notag
    \end{align}
    Further, at $(x_*,u_*)$, we have
    \begin{equation*}
        g_{\epsilon,\alpha,\gamma}(x_*,u_*)=\left\{\begin{aligned}&\argmin_{q}\text{ }\frac{1}{2}\|q +\epsilon\frac{\partial w}{\partial u}(u_*)^\top \nabla\Phi(x_*)\|^2\\
            &\quad\mathrm{s.t.:} \text{ } \nabla^\top b(u_*)\cdot q +\alpha b(u_*)\geq 0\\
            &\quad\quad\text{ }\begin{aligned} \underbrace{\frac{\partial h_r}{\partial u}(x_*,u_*)}_{\tfrac{\partial}{\partial u}\lie^r_{f}h(x_*,u_*)}\cdot q+\gamma h_{r}(x_*,u_*)\geq 0&\end{aligned}\end{aligned}\right.
    \end{equation*}
    Let $\tilde{f}:\real^m\to\real^n$ be defined by $\tilde{f}(u):=f(w(u),u)$.
    By definition of $w$, $\tilde{f}(u) = 0$ for all $u\in\real^m$.
    Thus $\dfrac{\partial \tilde{f}}{\partial u}(u)=0$ is equivalent to
    $\dfrac{\partial f}{\partial u}(w(u),u)=-\dfrac{\partial f}{\partial x}(w(u),u)\dfrac{\partial w}{\partial u}(u)$. By employing this relationship, we obtain
    \begin{equation}\label{eq:using_partial_of_w}
        \Big(\dfrac{\partial}{\partial u}\lie^r_{f}h(x_*,u_*)\Big)^\top = -\dfrac{\partial w}{\partial u}(u_*)^\top\Big(\dfrac{\partial f}{\partial x}(x_*,u_*)\Big)^r\nabla h(x_*) .
    \end{equation}
    
    The KKT conditions for the \eqref{eq:sgf_QP} QP are as follows
    \begin{equation}\label{eq:equiibria_proof_kkt_controller}
    \begin{aligned}
        &q_* +\epsilon\Big(\lambda_h\frac{\partial w}{\partial u}(u_*)^\top \nabla h(x_*) +\lambda_b\nabla b(u_*) \Big)- \lambda'_b\nabla b(u_*) +\\ &\qquad\qquad\qquad\quad\lambda'_h\dfrac{\partial w}{\partial u}(u_*)^\top\Big(\dfrac{\partial f}{\partial x}(x_*,u_*)\Big)^r\nabla h(x_*)= 0\\
        &0\leq\lambda'_b \perp \nabla^\top b(u_*)\cdot q_* +\alpha b(u_*)\geq 0\\ 
        &0\leq\lambda'_h \perp \frac{\partial \lie^r_{f}h}{\partial u}(x_*,u_*)\cdot q_*+\beta^{r}\gamma h(x_*)\geq0\\
    \end{aligned}
    \end{equation}
    where $\lambda'_h,\lambda'_b$ are Lagrange multipliers. Note that here, we have used \eqref{eq:equiibria_proof_kkt_feedback_opti} to replace $\frac{\partial w}{\partial u}(u_*)^\top \nabla\Phi(x_*)$ by  $\lambda_h\frac{\partial w}{\partial u}(u_*)^\top \nabla h(x_*) +\lambda_b\nabla b(u_*)$; employed \eqref{eq:using_partial_of_w}, and relied on \eqref{eq:h_i} and $f(x_*,u_*) = 0$ to deduce $h_r(x_*,u_*) = \beta h_{r-1}(x_*,u_*) = \beta^2 h_{r-2}(x_*,u_*)=\dots=\beta^r h(x_*)$. If we can find $\lambda'_h,\lambda'_b \geq 0$ such that \eqref{eq:equiibria_proof_kkt_controller} holds with $q_* = 0$, the statement is proven, as the KKT conditions are sufficient for global optimality of the QP. 
    \begin{itemize}
        \item If condition (1) holds, i.e., $h(x_*)>0$, then $\lambda_h = 0$. Picking $\lambda'_h = 0$, $\lambda'_b = \epsilon\lambda_b$, \eqref{eq:equiibria_proof_kkt_controller} is satisfied with $q_* = 0$.
        \item If any of (2) or (3) holds, then picking $\lambda'_h = -\frac{\epsilon}{e^r}\lambda_h$, $\lambda'_b = \lambda_b$, we see that \eqref{eq:equiibria_proof_kkt_controller} is satisfied with $q_* = 0$.
    \end{itemize}

    \textbf{Proof of $\impliedby$.} As $(x_*,u_*)$ is an equilibrium of \eqref{eq:closed_loop_sgf}, we have $f(x_*, u_*)=0$, and thus $x_* = w(u_*)$ and $g_{\epsilon,\alpha,\gamma}(x_*, u_*) = 0$. 
    Since either CRCQ or MFCQ holds at $(x_*,u_*)$,
    through the KKT conditions of the \eqref{eq:sgf_QP} QP, and noting that $q_* = g_{\epsilon,\alpha,\gamma}(x_*, u_*) = 0$, we obtain that there exist $\lambda'_h,\lambda'_b\geq 0$ that satisfy the following conditions
    \begin{equation}\label{eq:equilibria_proof_controller_kkt2}
    \begin{aligned}
        &\epsilon\frac{\partial w}{\partial u}(u_*)^\top \nabla\Phi(x_*)- \lambda'_b\nabla b(u_*) +\\ &\qquad\qquad\qquad\lambda'_h\dfrac{\partial w}{\partial u}(u_*)^\top\Big(\dfrac{\partial f}{\partial x}(x_*,u_*)\Big)^r\nabla h(x_*)= 0\\
        &0\leq\lambda'_b \perp \alpha b(u_*)\geq 0, \quad 0\leq\lambda'_h \perp \beta^{r}\gamma h(x_*)\geq0\\
    \end{aligned}
    \end{equation}
    where we used \eqref{eq:using_partial_of_w} and that $h_r(x_*,u_*) = \beta h_{r-1}(x_*,u_*) = \beta^2 h_{r-2}(x_*,u_*)=\dots=\beta^r h(x_*)$. To prove the statement, we need to find $\lambda_h,\lambda_b\geq 0$ and $\mu\in\real^n$ such that \eqref{eq:equiibria_proof_kkt_feedback_opti} holds. We pick $\mu = -\nabla\Phi(x_*) + \lambda_h \nabla h(x_*)$, and \eqref{eq:equiibria_proof_kkt_feedback_opti} becomes:
    \begin{equation*}
    \begin{aligned}
        &\begin{bmatrix}
            0 \\ 0 
        \end{bmatrix}  - \lambda_h \begin{bmatrix}
            0 \\ 0 
        \end{bmatrix} - \lambda_b \begin{bmatrix}
            0 \\ \nabla b(u_*) 
        \end{bmatrix} +\\ &\qquad\qquad\qquad\begin{bmatrix}
            0\\ -\dfrac{\partial w}{\partial u}(u_*)^\top(-\nabla\Phi(x_*) + \lambda_h \nabla h(x_*))
        \end{bmatrix} = 0\\
        &0\leq \lambda_h \perp h(x_*)\geq0, \quad 0\leq \lambda_b \perp b(u_*)\geq0, \quad x_* = w(u_*)
    \end{aligned}
    \end{equation*}
    Employing the first line of \eqref{eq:equilibria_proof_controller_kkt2}, the above becomes
    \begin{align}\label{eq:equiibria_proof_kkt_feedback_opti2}
        &-\lambda_b\nabla b(u_*) -\dfrac{\partial w}{\partial u}(u_*)^\top\lambda_h \nabla h(x_*) +\frac{\lambda'_b}{\epsilon}\nabla b(u_*) \notag
        \\
        &\qquad\qquad  -\frac{\lambda'_h}{\epsilon}\dfrac{\partial w}{\partial u}(u_*)^\top\Big(\dfrac{\partial f}{\partial x}(x_*,u_*)\Big)^r\nabla h(x_*) = 0
        \\
        &0\leq \lambda_h \perp h(x_*)\geq0, \quad 0\leq \lambda_b \perp b(u_*)\geq0, \quad x_* = w(u_*) . \notag 
    \end{align}
    Again, we have the following cases:
    \begin{itemize}
        \item If condition (1) holds, i.e., $h(x_*)>0$, then $\lambda'_h = 0$. Picking $\lambda_h = 0$, $\lambda_b = \frac{\lambda_b'}{\epsilon}$, we see that \eqref{eq:equiibria_proof_kkt_feedback_opti2} holds.
        \item If any of (2) or (3) holds, then picking $\lambda_h = \frac{-e^r\lambda'_h}{\epsilon}$, $\lambda_b = \frac{\lambda_b'}{\epsilon}$, we see that \eqref{eq:equiibria_proof_kkt_feedback_opti2} holds.
    \end{itemize}
    This completes the proof.
\end{proof}

\begin{proof}[\textbf{Proof of Proposition \ref{prop:regularization}}]
    Let us start by showing that here exists a sequence of points 
    $\{ (x^{(i)}, u^{(i)}) \}_{i\in\mathbb{N}}$ such that $\lim_{i\to\infty}(x^{(i)}, u^{(i)}) = (x_*,u_*)$
    and $h( w(u^{(i)}) ) > 0$, $b(u^{(i)}) \geq 0$ for all $i\in\mathbb{N}$.
    Indeed, consider the case when $b(u_*) > 0$. By MFCQ, 
    there exists $v_1\in\real^n, \tilde{v}_1\in\real^m$ such that $\nabla h(x_*)^\top v_1 > 0$
    and $v_1 = \frac{\partial w}{\partial u}(u_*) \tilde{v}_1$ (this last equality follows from MFCQ applied to the equality constraint $x = w(u)$). Hence, since $h(w( u_* + s \tilde{v}_1 )) = h(w(u_*)) + s \tilde{v}_1^\top \frac{\partial w}{\partial u}(u_*)^\top \nabla h( w(u_*) ) + \mathcal{O}(s^2) = 
    h(x_*) + s v_1^\top \nabla h( x_* ) + \mathcal{O}(s^2)
    $,
    there exists $s_{*,1} > 0$ sufficiently small such that $h( w(u_* + s \tilde{v}_1) ) > 0$ and $b(u_* + s v_1) > 0$ for all $s \leq s_{*,1}$.
    Alternatively, if $b(u_*) = 0$, by MFCQ, there exist $v_2\in\real^n$ and $v_3\in\real^m$ such that $\nabla h(x_*)^\top v_2 > 0$ and $\nabla b(u_*)^\top v_3 > 0$. 
    Now, by a similar argument as before
    there exists $s_{*,2} > 0$ sufficiently small such that $h(w(u_* + sv_2)) > 0$ and $b(u_* + s v_3) > 0$ for all $s \leq s_{*,2}$.
    In either case, by considering sufficiently small values of $s$, there exists a sequence of points 
    $\{ (x^{(i)}, u^{(i)}) \}_{i\in\mathbb{N}}$ such that $\lim_{i\to\infty}(x^{(i)}, u^{(i)}) = (x_*,u_*)$
    and $h( w(u^{(i)}) ) > 0$, $b(u^{(i)}) \geq 0$ for all $i\in\mathbb{N}$.
    Since $h$, $w$ and $\Phi$ are continuous,
    for any $c>0$, there exists $u_{c}\in\real^m$ such that $h(w(u_{c})) > 0$, $b(u_{c}) \geq 0$ and $|\Phi( w(u_{c}) )-\Phi( w(u_*) )| < c$. Now, let $\varepsilon := h( w(u_{c}) )$ and $p > 0$ such that 
    $\Phi(w(u_{c})) < \Phi(w(u_*)) + p\varepsilon^2$. Note that $(w(u_c),u_{c})$ is feasible for~\eqref{eq:perturbed-optimization-problem}
    and the value of the objective function is $\Phi( w(u_{c}) )$.
    On the other hand, since $h(w(u_*)) = 0$, 
    the optimizer of
    \begin{align}\label{eq:perturbed-optimization-problem-2}
    \notag
    &\min\limits_{u\in\real^m} \Phi(w(u)) + p (\varepsilon-h(w(u)))^2 \\
    &\quad\text{s.t.} \quad b(u) \geq 0, \ h(w(u)) = 0.
    \end{align}
    is $u_*$. This means that for any $u\in\real^m$ such that $h(w(u)) = 0$, the value of the cost function in~\eqref{eq:perturbed-optimization-problem} is at least $\Phi(w(u_*)) + p\varepsilon^2$.
    Since $\Phi(w(u_*)) + p\varepsilon^2 > \Phi(w(u_{c}))$, it follows that the optimizer $u_*'$ of~\eqref{eq:perturbed-optimization-problem}  can not satisfy $h(w(u_*')) = 0$.
    Now, since $u^{\prime}_*$ is the optimizer of~\eqref{eq:perturbed-optimization-problem},
    we have
    \begin{align*}
        \Phi(w(u_*')) + p (\varepsilon-h(w(u_*')))^2 \leq \Phi( w( u_{c} ) ).
    \end{align*}
    This implies that $\Phi(w(u_*')) \leq \Phi( w( u_{c} ) )$
    and we have
    \begin{align*}
        \Phi(w( u_*) ) \leq \Phi( w(u_*') ) \leq \Phi( w(u_{c}) ).
    \end{align*}
    Since $|\Phi( w(u_{c}) )-\Phi( w(u_*) )| < c$, this in turn implies that 
    $|\Phi( w(u_*') )-\Phi( w(u_*) )| < c$.
\end{proof}

\begin{proof}[\textbf{Proof of Theorem \ref{thm:local_convergence}}]
    Consider the function $V:\Nc\to\real$ defined by $V(x,u):= \Phi(w(u))-\Phi(w(u_*))+W(x,u)$, where $W$ is given by Lemma \ref{lem:W}. Let us show that $V$ is a Lyapunov function. Note that $V(x_*,u_*) = 0$, as $x_*=w(u_*)$. Moreover, since
    $\Phi(w(u))>\Phi(w(u_*))$ for all $(x,u)\in\Nc \cap\Sc_r \setminus\{(x_*,u_*)\}$,
    $V(x,u) > 0$ for all $(x,u)\in \Nc \cap\Sc_r \setminus\{ (x_*,u_*) \}$. 
    Note that
    \begin{equation*}
       \frac{d}{dt} \Phi(w(\upsilon(t)))=  \nabla^\top\Phi(w(\upsilon(t))) \frac{\partial w}{\partial u}(\upsilon(t)) g_{\epsilon,\alpha,\gamma}(\xi(t),\upsilon(t)),
    \end{equation*}
    where $(\xi(t),\upsilon(t))$ is a solution to \eqref{eq:closed_loop_sgf}, with initial condition $(\xi(0),\upsilon(0))\in \Sc_r$. 
    Notice that $b(\upsilon(t))\geq 0$ and $h_r(\xi(t),\upsilon(t))\geq 0$ for all time, from forward invariance of $\Sc_r$ (see the proof of Proposition \ref{prop:safety}). In the following, we drop the dependence of $\xi$ and $\upsilon$ on $t$ for convenience.
    Let $\Gamma\subseteq \Nc\cap\Sc_r$ be a sublevel set of $V$ such that $h_r(x,u) > 0$ for all $(x,u)\in\Gamma$.\footnote{Such a sublevel set exists because: a) $V$ has compact sublevel sets due to $\Phi$ having compact sublevel sets and $W$ being radially unbounded; and b) $(x_*,u_*)$ satisfies $h_r(x_*,u_*)=\beta^r h(x_*)>0$, and due to continuity of $h_r$, $\Gamma$ can be picked small enough so that $h_r(x,u)>0$ for all $(x,u)\in\Gamma$.} Since $\nabla \Phi$ and $\frac{\partial w}{\partial u}$ are locally Lipschitz, and $\Gamma$ is compact, there exists $L>0$ such that for all $(x,u),(x^\prime,u)\in\Gamma$, we have
    $\|\nabla^\top \Phi(x)\frac{\partial w}{\partial u}(u)-\nabla^\top \Phi(x')\frac{\partial w}{\partial u}(u)\|\leq L\|x-x'\|$. Then, whenever $(\xi(t),\upsilon(t))\in\Gamma$, we have
    \begin{align}\label{eq:giannis1}
         &\frac{d}{dt} \Phi(w(\upsilon)) \notag
        \\
        &= \big(\nabla^\top\Phi(w(\upsilon))\cdot\frac{\partial w}{\partial u}(\upsilon)-\nabla^\top\Phi(\xi)\cdot\frac{\partial w}{\partial u}(\upsilon)\big)g_{\epsilon,\alpha,\gamma}(\xi,\upsilon) \notag
        \\
        & \quad + \nabla^\top\Phi(\xi)\cdot\frac{\partial w}{\partial u}(\upsilon)g_{\epsilon,\alpha,\gamma}(\xi,\upsilon)
        \\
        & \leq
        L \|\xi-w(\upsilon)\|\|g_{\epsilon,\alpha,\gamma}(\xi,\upsilon)\|
         +\nabla^\top\Phi(\xi)\cdot\frac{\partial w}{\partial u}(\upsilon)g_{\epsilon,\alpha,\gamma}(\xi,\upsilon). \notag
    \end{align}
    Now, from the KKT conditions of~\eqref{eq:sgf_QP}, we have
   \begin{align}\label{eq:giannis_kkt1}
        g_{\epsilon,\alpha,\gamma}(x,u) +  \epsilon\frac{\partial w}{\partial u}(u)^\top\nabla\Phi(x)&-\lambda'_b(x,u)\nabla b(u) \\&- \lambda'_h(x,u) \frac{\partial h_r}{\partial u}(x,u)^\top = 0  \notag
    \end{align}
    where $\lambda'_b,\lambda'_h$ are Lagrange multipliers satisfying
\begin{align}\label{eq:giannis_kkt2}
        &0\leq\lambda'_b(x,u)\perp \nabla^\top b(u)g_{\epsilon,\alpha,\gamma}(x,u) + \alpha b(u)\geq 0, 
        \\
        &0\leq\lambda'_h(x,u)\perp \frac{\partial h_r}{\partial u}(x,u)g_{\epsilon,\alpha,\gamma}(x,u) + \frac{\partial h_r}{\partial x}(x,u)f(x,u) \notag
        \\
        &\qquad\qquad\qquad\qquad\qquad\qquad\qquad\qquad+\gamma h_r(x,u)\geq 0 . \notag
    \end{align}
    Using these, for the second term of \eqref{eq:giannis1}, we obtain
    \begin{align}\label{eq:giannis2}
    & \nabla^\top\Phi(\xi)\cdot\frac{\partial w}{\partial u}(\upsilon)g_{\epsilon,\alpha,\gamma}(\xi,\upsilon) = \frac{1}{\epsilon} \Big(\lambda'_b(\xi,\upsilon)\nabla b(\upsilon) \notag
    \\
    & \quad -g_{\epsilon,\alpha,\gamma}(\xi,\upsilon) + \lambda'_h(\xi,\upsilon)\frac{\partial h_r}{\partial u}(\xi,\upsilon(t))\Big)^\top g_{\epsilon,\alpha,\gamma}(\xi,\upsilon)
    \notag
    \\
    &=- \frac{1}{\epsilon} \|g_{\epsilon,\alpha,\gamma}(\xi,\upsilon)\|^2 - \frac{1}{\epsilon}\lambda'_b(\xi,\upsilon)\alpha b(\upsilon) 
    \\
    & \quad -
    \frac{1}{\epsilon}\lambda'_h(\xi,\upsilon)\gamma h_r(\xi,\upsilon)-\frac{1}{\epsilon}\lambda'_h(\xi,\upsilon)\frac{\partial h_r}{\partial x}(\xi,\upsilon)f(\xi,\upsilon), \notag
    \end{align} 
    whenever $(\xi(t),\upsilon(t))\in\Gamma$, where in the first step we used \eqref{eq:giannis_kkt1} and in the second step we used \eqref{eq:giannis_kkt2}.
    Now, since $h_r(x,u) > 0$ for all $(x,u)\in\Gamma$, and $\Gamma$ is compact,
    there exists $c_0>0$ such that $h_r(x,u)\geq c_0$
    for all $(x,u)\in\Gamma$. Moreover, since $\Gamma$ is compact and $\frac{\partial h_r}{\partial x}$ and $f$ are continuous, there exists $c_1>0$ such that 
    $\frac{\partial h_r}{\partial x}(x,u)f(x,u) \geq -c_1$ for all $(x,u)\in\Gamma$. Hence, by taking $\gamma \geq \gamma_* := \frac{c_1}{c_0}$, 
    we have $\gamma h_r(x,u) + \frac{\partial h_r}{\partial x}(x,u)f(x,u) \geq 0$ for all $(x,u)\in\Gamma$, which, using~\eqref{eq:giannis2} implies that 
    \begin{align*}
        \nabla^\top\Phi(\xi)\cdot\frac{\partial w}{\partial u}(\upsilon)g_{\epsilon,\alpha,\gamma}(\xi,\upsilon) \leq -\frac{1}{\epsilon}\|g_{\epsilon,\alpha,\gamma}(\xi,\upsilon)\|^2,
    \end{align*}
    whenever $(\xi,\upsilon)\in\Gamma$,
    where we have also used the fact that $b(\upsilon(t)) \geq 0$ for all $t\geq0$.
    Note also that, from Lemma~\ref{lem:W},
    \begin{equation*}
        \begin{aligned}
            \frac{d}{dt}W(\xi,\upsilon) \! \leq \! & - \! d_3 \! \norm{\xi-w(\upsilon)}^2 \! + \! d_4 \! \norm{ g_{\epsilon,\alpha,\gamma}(\xi,\upsilon)} \norm{\xi-w(\upsilon)}.
        \end{aligned}
    \end{equation*}
    Hence, 
    whenever $(\xi(t),\upsilon(t))\in\Gamma$, we have $\frac{d}{dt}V(\xi(t),\upsilon(t)) \leq \zeta(t)^T P_{\epsilon} \zeta(t)$, where $\zeta(t) = [ \norm{g_{\epsilon,\alpha,\gamma}(\xi(t),\upsilon(t))}, \norm{\xi(t)-w(\upsilon(t))} ]$
    and 
    \begin{align*}
        P_{\epsilon} = 
        \begin{pmatrix}
            -\frac{1}{\epsilon} & \frac{L+d_4}{2} \\
            \frac{L+d_4}{2} & -d_3
        \end{pmatrix} .
    \end{align*}
    Note that $\text{det}(P_{\epsilon}) = \frac{d_3}{\epsilon}-\big( \frac{L+d_4}{2} \big)^2$ and hence, by taking $\epsilon < \epsilon_* := \frac{4d_3}{ (L+d_4)^2 }$, $P_{\epsilon}$ is negative definite. Therefore, by taking $\epsilon\in(0,\epsilon_*)$ and $\gamma\geq \gamma_*$, $\frac{d}{dt}V(\xi(t),\upsilon(t)) < 0$ whenever $(\xi(t),\upsilon(t))\in\Gamma$, unless $g_{\epsilon,\alpha,\gamma}(\xi,\upsilon)=0$ and $\xi-w(\upsilon)=0$. Now, for any point $(x,u)\in \Gamma$, such that $x=w(u)$, we have that $h(x) = \frac{1}{\beta^r}h_r(x,u)>0$. Thus, if $g_{\epsilon,\alpha,\gamma}(x,u)=0$ and $x=w(u)$,  from Proposition \ref{prop:critical_points_equilibria}, we have that $(x,u)$ is a critical point of \eqref{eq:steady-state_opti_problem}; since $(x_*,u_*)$ is the sole critical point of \eqref{eq:steady-state_opti_problem}, $\frac{d}{dt}V(\xi(t),\upsilon(t)) < 0$ for all $(\xi(t),\upsilon(t))\in\Gamma\setminus\{(x_*,u_*)\}$, and $V(x_*,u_*) = 0$. This completes the proof.
\end{proof}
\begin{proof}[\textbf{Proof of Lemma \ref{lem:conditions-sigma-strong-mfcq}}]
    For each $(x,u)\in\Cc_r$ such that $b(u) = 0$, note that by the assumptions of Proposition~\ref{prop:feasibility}, we have 
    $\nabla b(u)^\top q_{(x,u)} > 0$, and there exists a neighborhood $\Nc_{(x,u)}$ of $(x,u)$ and a constant $\eta_{(x,u)} > 0$ such that
    \begin{align*}
        \eta_{(x,u)} < \frac{\nabla b(\bar{u})^\top q_{(x,u)}}{ \norm{\nabla b(\bar{u})} \norm{q_{(x,u)}} },
    \end{align*}
    for all $(\bar{x},\bar{u})\in\Nc_{(x,u)}$. Note that the denominator of the above expression is non-zero because otherwise $\nabla b(u)^\top q_{(x,u)} > 0$ would not hold.
    Now, define the set $B_1 = \setdef{(x,u)\in\Cc_r}{b(u) = 0}$, and note that $\bigcup_{(x,u)\in B_1} \Nc_{(x,u)}$ is a covering of $B_1$. Since $\Cc_r$ is compact, $B_1$ is also compact and it admits a finite cover $\bigcup_{i=1}^N \Nc_{(x_i,u_i)}$.
    Now, by letting $\sigma_1 := \min\limits_{i\in[N]} \eta_{(x_i,u_i)}$, we have that for each $(x,u)\in B_1$, there exists $i\in[N]$ such that 
    $\nabla b(u)^\top q_{(x_i,u_i)} \geq \sigma_1 \norm{q_{(x_i,u_i)}} \norm{\nabla b(u)}$, and the first item holds.    
    The second item follows by an analogous argument.
\end{proof}

\begin{proof}[\textbf{Proof of Theorem \ref{thm:global_convergence}}]
    In what follows, we assume without loss of generality that $\epsilon\leq \epsilon_m$, for some arbitrary $\epsilon_m>0$. Further, notice that $(x_*,u_*) \in \bigcap_{i=0}^r\Sc_i$, by Proposition~\ref{prop:feasible-points-lie-in-bigcap}; that solutions of \eqref{eq:closed_loop_sgf} exist and are unique, by Proposition \ref{prop:solutions_sgf}; and that CRCQ holds for \eqref{eq:sgf_QP} in $\bigcap_{i=0}^r\Sc_i$, by Proposition~\ref{prop:crcq}. 

    \emph{$\bullet$ \underline{Defining the function $V$ and its time derivative}:}  
    Consider the function $V:\real^m\to\real$ defined by $V(u) = \frac{1}{2}(u-u_*)^\top (u-u_*)$.
    We have 
        \begin{align}\label{eq:dVdt_1}
        & \frac{d}{dt}V(x,u) = (u-u_*)^\top g_{\epsilon,\alpha,\gamma}(x,u) = \notag
        \\
        & =\epsilon(u-u_*)^\top \Big(-\frac{\partial w}{\partial u}(u)^\top\nabla \Phi(w(u))\Big)  \notag
        \\
        & \quad + \epsilon(u-u_*)^\top \! \Big(-\frac{\partial w}{\partial u}(u)^\top\nabla\Phi(x) + \frac{\partial w}{\partial u}(u)^\top\nabla\Phi(w(u))\Big) \notag
        \\
        & \quad +
        (u-u_*)^\top \! \Big(\lambda_b'(x,u)\nabla b(u) \!+\! \lambda_h'(x,u)\frac{\partial h_r}{\partial u}(x,u)\Big),
    \end{align}
    where the derivative w.r.t. time is taken over the closed-loop dynamics \eqref{eq:closed_loop_sgf}, and in the second equality we have used the 
    KKT conditions~\eqref{eq:giannis_kkt1}-\eqref{eq:giannis_kkt2} for the QP \eqref{eq:sgf_QP}.
    
    Now, note that since $b(u_*) > 0$ and $h_r(x_*,u_*) > 0$, the stationarity condition for~\eqref{eq:steady-state_opti_problem} reads $\frac{\partial w}{\partial u}(u_*)^\top\nabla\Phi(w(u_*)) = \mathbf{0}_m$. Due to strong convexity of $\Phi(w(\cdot))$, we have monotonicity of its gradient, i.e., we have $-(u-u_*)^\top (\frac{\partial w}{\partial u}(u)^\top\nabla\Phi(w(u)) - \frac{\partial w}{\partial u}(u_*)^\top\nabla\Phi(w(u_*))) \leq -c \|u-u_*\|^2$. This provides a bound on the first term of the RHS of \eqref{eq:dVdt_1}. Further, notice that, due to Assumption \ref{assum:differentiability, regularity and relative degree} and the fact that $\bigcap_{i=0}^r\Sc_i$ is compact, there exists $L_1>0$ such that 
    \begin{equation*}
        \Big\| \frac{\partial w}{\partial u}(u)^\top \nabla \Phi(x)- \frac{\partial w}{\partial u}(u)^\top \nabla \Phi(x') \Big\| \leq L_1\|x-x'\|
    \end{equation*} for any $(x,u),(x',u)\in \bigcap_{i=0}^r\Sc_i$. This provides a bound for the second term of the RHS of \eqref{eq:dVdt_1}. Overall, we obtain:
    \begin{align}\label{eq:dVdt_2}
        & \frac{d}{dt}V(x,u) \leq -\epsilon c \norm{u - u_*}^2 \!+\! \epsilon L_1\norm{u-u_*} \norm{x-w(u)}         \notag
        \\
        & \quad + \! (u-u_*)^\top \Big(\lambda_b'(x,u)\nabla b(u) \!+\! \lambda_h'(x,u)\frac{\partial h_r}{\partial u}(x,u) \! \Big).
    \end{align}

    \emph{$\bullet$ \underline{Upper bounding remaining term in time derivative of $V$}:} 
    Now, let us derive an upper bound for the term 
    $(u-u_*)^\top \big(\lambda_b'(x,u)\nabla b(u) \!+\! \lambda_h'(x,u)\frac{\partial h_r}{\partial u}(x,u) \! \big)$. Choose $\delta>0$ such that $b(u_*) > \delta$ and $h_r(x,u_*)>\delta$ for all $x$ such that $(x,u^*)\in\bigcap_{i=0}^r\Sc_i$. Note that such $\delta$ exists by assumption. 
    Define the function 
    \begin{equation}\label{eq:definition-gb}
    g_{b_{\epsilon,\alpha}}(x,u)=\left\{\begin{aligned}\argmin_{q}\text{ }& \frac{1}{2}\|q +\epsilon\frac{\partial w}{\partial u}(u)^\top \nabla\Phi(x)\|^2 \\
            \mathrm{s.t.:} \text{ } &\nabla^\top b(u)\cdot q +\alpha b(u)\geq 0.
    \end{aligned}\right.
    \end{equation}
    For all $(x,u)\in\bigcap_{i=0}^r\Sc_i$, $q=0$ is feasible for~\eqref{eq:definition-gb}.
    Therefore,
    \begin{align*}
         \| g_{b_{\epsilon,\alpha}}(x,u) \| - \epsilon \| \frac{\partial w}{\partial u}(u)^\top \nabla\Phi(x) \|  & \leq
        \\
        \| g_{b_{\epsilon,\alpha}}(x,u)+\epsilon\frac{\partial w}{\partial u}(u)^\top \nabla\Phi(x)\| & \leq \| \epsilon\frac{\partial w}{\partial u}(u)^\top \nabla\Phi(x)\|,
    \end{align*}
    where in the first inequality we used the triangle inequality and in the second inequality the fact that $g_{b_{\epsilon,\alpha}}$ is the minimizer of~\eqref{eq:definition-gb}.
    We hence get
    \begin{align}\label{eq:gb-bound}
        \| g_{b_{\epsilon,\alpha}}(x,u) \| \leq 2\epsilon \Big\| \frac{\partial w}{\partial u}(u)^\top \nabla\Phi(x) \Big\|.
    \end{align}
    Since $\bigcap_{i=0}^r\Sc_i$ is compact and $\epsilon\leq \epsilon_m$,~\eqref{eq:gb-bound} implies that $g_{b_{\epsilon,\alpha}}$ is uniformly bounded, i.e., there exists $M_b>0$ (independent of $\alpha$) such that $\norm{g_{b_{\epsilon,\alpha}}(x,u)} \leq M_b$ for all $(x,u)\in\bigcap_{i=0}^r\Sc_i$. 
    Hence, there exists $\bar{M}_b>0$ such that, for all $(x,u)\in\bigcap_{i=0}^r\Sc_i$, 
    \begin{align*}
        \Big\| \frac{\partial h_r}{\partial u}(x,u)^\top g_{b_{\epsilon,\alpha}}(x,u) + \frac{\partial h_r}{\partial x}(x,u)^\top f(x,u) \Big\| \leq \bar{M}_b.
    \end{align*}
    Therefore, by taking $\gamma > \gamma_1 := \frac{\bar{M}_b}{\delta}$, we have 
    \begin{align}\label{eq:condition-h}
        \frac{\partial h_r}{\partial u}(x,u)^\top g_{b_{\epsilon,\alpha}}(x,u) + \frac{\partial h_r}{\partial x}(x,u)^\top f(x,u) + \gamma h_r(x,u) > 0,
    \end{align}
    for all $(x,u)\in\bigcap_{i=0}^r\Sc_i$ with $h_r(x,u)\geq\delta$. Similarly, define
    \begin{align}\label{eq:definition-gh}
    g_{h_{\epsilon,\gamma}}(x,u)=\left\{\begin{aligned}\argmin_{q}\text{ }& \frac{1}{2}\|q +\epsilon\frac{\partial w}{\partial u}(u)^\top \nabla\Phi(x)\|^2 \\
            \mathrm{s.t.:} \text{ } &\begin{aligned}&\\\frac{\partial h_r}{\partial x}(x,u)\cdot f(x,u)+ \frac{\partial h_r}{\partial u}(x,u)\cdot q&
            \\+\gamma h_{r}(x,u)\geq 0&\end{aligned}
            \end{aligned}\right.
    \end{align}
    Note that, since $\frac{\partial h_r}{\partial u}(x,u)\neq 0$ for all $(x,u)$ such that $h_r(x,u)$ by Assumption~\ref{assum_item:relative degree}, and since $\bigcap_{i=0}^r\Sc_i$ is compact, using arguments analogous to the proof of Proposition~\ref{prop:feasibility}, there exists a neighborhood $\Nc\subseteq\bigcap_{i=0}^r\Sc_i$ of $\{(x,u):\ h_r(x,u)=0\}\cap\bigcap_{i=0}^r\Sc_i$, where $\frac{\partial h_r}{\partial u}(x,u)\neq 0$. Then, for any $(x,u)\in \Nc$, we have
    \begin{align*}
        q = -\frac{ \frac{\partial h_r}{\partial x}(x,u)^\top f(x,u) }{ \norm{\frac{\partial h_r}{\partial u}(x,u)}^2 } \frac{\partial h_r}{\partial u}(x,u),
    \end{align*}
    is feasible for~\eqref{eq:definition-gh}.
    Hence, for any $(x,u)\in \Nc$,
    \begin{align*}
        & \Big\| g_{h_{\epsilon,\gamma}}(x,u) \Big\| - \epsilon \Big\| \frac{\partial w}{\partial u}(u)^\top \nabla\Phi(x) \Big\| \leq
        \\ 
        & \Big\| g_{h_{\epsilon,\gamma}}(x,u) + \epsilon \frac{\partial w}{\partial u}(u)^\top \nabla\Phi(x) \Big\| 
        \leq
        \\ 
        & \Big\|-\frac{ \frac{\partial h_r}{\partial x}(x,u)^\top f(x,u) }{ \norm{\frac{\partial h_r}{\partial u}(x,u)}^2 } \frac{\partial h_r}{\partial u}(x,u) + \epsilon \frac{\partial w}{\partial u}(u)^\top \nabla\Phi(x) \Big\|,
    \end{align*}
    which implies that 
        \begin{align}\label{eq:gh-bound}
        & \norm{g_{h_{\epsilon,\gamma}}(x,u)} \leq \epsilon \Big\| \frac{\partial w}{\partial u}(u)^\top \nabla\Phi(x) \Big\|
        \\ 
        & \; + \Big\| -\frac{ \frac{\partial h_r}{\partial x}(x,u)^\top f(x,u) }{ \norm{\frac{\partial h_r}{\partial u}(x,u)}^2 } \frac{\partial h_r}{\partial u}(x,u) + \epsilon \frac{\partial w}{\partial u}(u)^\top \nabla\Phi(x) \Big\|, \notag
    \end{align}
    for any $(x,u)\in \Nc$.
    Further, again using arguments analogous to the proof of Proposition~\ref{prop:feasibility}, there exists $\gamma_2$, such that for $\gamma>\gamma_2$ we have that, for all $(x,u)\in (\bigcap_{i=0}^r\Sc_i)\setminus \Nc$, $q=0$ is feasible for \eqref{eq:definition-gh}. Thus, similarly to the above, for any $(x,u)\in (\bigcap_{i=0}^r\Sc_i)\setminus \Nc$, we have $\|g_{h_{\epsilon,\gamma}}(x,u)\|\leq 2\epsilon \norm{\frac{\partial w}{\partial u}(u)^\top \nabla\Phi(x)}$, when $\gamma>\gamma_2$.
    
    Since $\bigcap_{i=0}^r\Sc_i$ is compact and $\epsilon\leq\epsilon_m$, the above implies that, when $\gamma>\gamma_2$ $g_{h_{\epsilon,\gamma}}$ is uniformly bounded in $\bigcap_{i=0}^r\Sc_i$, i.e., there exists $M_h > 0$ (independent of $\gamma$) such that $\norm{g_{h_{\epsilon,\gamma}}(x,u)} \leq M_h$ for all $(x,u)\in\bigcap_{i=0}^r\Sc_i$. In turn, this implies that there exists $\bar{M}_h>0$ such that, when $\gamma>\gamma_2$, $\norm{\nabla b(u)^\top g_{h_{\epsilon,\gamma}}(x,u)} \leq \bar{M}_h$.
    Therefore, by taking $\alpha > \alpha_1 := \frac{\bar{M}_h}{\delta}$ and $\gamma>\gamma_2$, we have 
    \begin{equation}
        \begin{aligned}\label{eq:condition-b}
        &\nabla b(u)^\top g_{h_{\epsilon,\gamma}}(x,u) + \alpha b(u) > 0,
        \end{aligned}
    \end{equation}
    for all $(x,u) \in \bigcap_{i=0}^r\Sc_i \ \text{s.t.} \ b(u) \geq \delta$.
    Finally, let $\tilde{M}_b > 0$, $\tilde{M}_h > 0$ such that 
    \begin{align*}
        &\big\|\nabla b(u)^\top \epsilon_m \frac{\partial w}{\partial u}(u)^\top \nabla\Phi(x) \big\| \leq \tilde{M}_b,
        \\
        &\big\|\frac{\partial h_r}{\partial u}(x,u)^\top \epsilon_m \frac{\partial w}{\partial u}(u)^\top \! \nabla\Phi(x)\big\| \!+\!
\big\|\frac{\partial h_r}{\partial x}(x,u)^\top \! f(x,u) \big\| \!\leq \! \tilde{M}_h,
    \end{align*}
    for all $(x,u)\in\bigcap_{i=0}^r \Sc_i$. By taking $\alpha > \alpha_2 = \frac{\tilde{M}_b}{\delta}$ and $\gamma >\gamma_3= \frac{\tilde{M}_h}{\delta}$, we have 
    $g_{\epsilon,\alpha,\gamma}(x,u) = -\epsilon \frac{\partial w}{\partial u}(u)^\top \nabla\Phi(x)$,
    for all $(x,u)$ such that $b(u) \geq \delta$ and $h_r(x,u) \geq \delta$,
    and the constraints of~\eqref{eq:sgf_QP} are not active at the optimizer.
    
    Now, let $\alpha > \alpha^*_1 := \max\{ \alpha_1, \alpha_2 \}$, $\gamma > \gamma^*_1 := \max\{ \gamma_1, \gamma_2, \gamma_3 \}$. Using~\eqref{eq:dVdt_2} and the observations above, we next show that the following holds, for all $(x,u)\in\bigcap_{i=0}^r \Sc_i$:
{      \small    \begin{align}\label{eq:global_proof_bound_dvdt}
  \frac{d}{dt}V(x,u) \leq -\epsilon c\norm{u-u_*}^2
        + {\epsilon L_1} \norm{u-u_*} \norm{x-w(u)}.
    \end{align}
    }
    For $(x,u)\in\bigcap_{i=0}^r \Sc_i$, consider the following four cases:
    \newline
    \mbox{} \quad 1)
 if $b(u)\geq\delta$ and $h_r(x,u)\geq\delta$, then since the constraints of~\eqref{eq:sgf_QP} are not active at the optimizer, by the complementary slackness condition, we have $\lambda_b^{\prime}(x,u) = 0$ and $\lambda_h^{\prime}(x,u) = 0$, which implies \eqref{eq:global_proof_bound_dvdt}.
        \newline
        \mbox{} \quad 2) if $b(u)\in[0,\delta)$ and $h_r(x,u)\geq\delta$, then since~\eqref{eq:condition-h} holds, $g_{\epsilon,\alpha,\gamma}(x,u) = g_{b_{\epsilon,\alpha}}(x,u)$ and the constraint in~\eqref{eq:sgf_QP} associated with $h_r$ is not active.\footnote{That is because, as $g_{b_{\epsilon,\alpha}}(x,u)$ is the minimizer of \eqref{eq:definition-gb}, and because it is feasible for \eqref{eq:sgf_QP}, then it has to be a minimizer of \eqref{eq:sgf_QP} (as \eqref{eq:sgf_QP} is more constrained than \eqref{eq:definition-gb}).} Hence, $\lambda_h^{\prime}(x,u) = 0$. On the other hand, 
        since $-b$ is convex, we have 
        $-b(u_*) \geq -b(u) - \nabla b(u)^\top (u^* - u)$. Since $b(u_*) > \delta \geq b(u)$, we have 
        $0 \geq b(u) - b(u_*) \geq \nabla b(u)^\top (u-u^*)$. Furthermore, given that $\lambda_b^{\prime}(x,u) \geq 0$, this implies that $(u-u_*)^\top \lambda_b^{\prime}(x,u) \nabla b(u) \leq 0$.
        Altogether, we have \eqref{eq:global_proof_bound_dvdt}.
        \newline
        \mbox{} \quad 3) if $b(u)\geq\delta$ and $h_r(x,u)\in[0,\delta)$, then since~\eqref{eq:condition-b} holds, $g_{\epsilon,\alpha,\gamma}(x,u) = g_{h_{\epsilon,\gamma}}(x,u)$ and the constraint in~\eqref{eq:sgf_QP} associated with $b$ is not active.\footnote{The reasoning is the same as in the previous case.} Hence, $\lambda_b^{\prime}(x,u) = 0$. On the other hand, since $-h_r(x,\cdot)$ is convex for any $x\in\real^n$ such that $\exists u\in\real^m$ with $(x,u)\in\bigcap_{i=0}^r \Sc_i$, $-h_r(x,u_*) \geq -h_r(x,u) - \frac{\partial h_r}{\partial u}(x,u)^\top (u_*-u)$. Since $h_r(x,u_*) > \delta \geq h_r(x,u)$ for any $x$ such that there exists $u$ with $(x,u)\in\bigcap_{i=0}^r \Sc_i$, we have 
        $0 \geq h_r(x,u)-h_r(x,u_*) \geq \frac{\partial h_r}{\partial u}(x,u)^\top (u-u_*)$. Since $\lambda_h^{\prime}(x,u) \geq 0$, we have $(u-u_*)^\top \lambda_h^{\prime}(x,u) \frac{\partial h_r}{\partial u}(x,u) \leq 0$, and~\eqref{eq:global_proof_bound_dvdt} follows.
        \newline
        \mbox{} \quad
        4) 
        if $b(u)\in[0,\delta)$ and $h_r(x,u)\in[0,\delta)$, by the same arguments utilized in the previous two items we have $(u-u_*)^\top \lambda_b^{\prime}(x,u) \nabla b(u) \leq 0$ and $(u-u_*)^\top \lambda_h^{\prime}(x,u) \frac{\partial h_r}{\partial u}(x,u) \leq 0$, which implies~\eqref{eq:global_proof_bound_dvdt}.
 
    As a result, we conclude that \eqref{eq:global_proof_bound_dvdt} holds on $\bigcap_{i=0}^r \Sc_i$.

    \emph{$\bullet$ \underline{Defining the Lyapunov function $\tilde{V}$ and its time derivative}:}    Next, consider the function $W:\real^n\times\real^m\to\real$, from Lemma \ref{lem:W}. Our strategy is to use the combined Lyapunov function $\tilde{V} = V+W$ to prove global stability
    of the closed-loop. For the function $W$ we have
    \begin{align*}
        \frac{d}{dt}W(x,u) \! \leq  \! d_4 \norm{ g_{\epsilon,\alpha,\gamma}(x,u)} \norm{x-w(u)} \!- \! d_3 \norm{x-w(u)}^2.
    \end{align*}
    Combining this with \eqref{eq:global_proof_bound_dvdt}, we obtain:
    \begin{align}\label{eq:dvtilde_dt}
        & \frac{d}{dt}\tilde{V}(x,u) \leq
         -\epsilon c\norm{u-u_*}^2 + \epsilon L_1 \norm{u-u_*} \norm{x-w(u)} \notag
        \\
        & - 
        d_3 \norm{x-w(u)}^2 + d_4 \norm{ g_{\epsilon,\alpha,\gamma}(x,u)} \norm{x-w(u)} .
    \end{align}
    Below, we show that there exist $\epsilon^* > 0$, $\alpha^*>0$ and $\gamma^* >0$ such that if $\epsilon < \epsilon^*$, $\alpha > \alpha^*$, $\gamma > \gamma^*$, then 
    $\frac{d \tilde{V}}{dt}(x,u) < 0$ for all $(x,u)\in (\bigcap_{i=0}^r \Sc_i) \setminus (x_*,u_*)$.

    Let $\bar{c} > 0$ be sufficiently small such that 
    \begin{align*}
        B_{\bar{c}} := \setdef{ (x,u) \! \in \! \bigcap_{i=0}^r \Sc_i }{ \norm{ (x \! - \! w(u),u \! - \! u_*) } \! \leq \! \bar{c} } \! \subseteq \! \text{Int}(\bigcap_{i=0}^r \Sc_i).
    \end{align*}
    Note that such $\bar{c}$ exists because $(x_*,u_*)\in\text{Int}(\bigcap_{i=0}^r \Sc_i)$.\footnote{As $h_r(x_*,u_*) >0$ and $x_*=w(u_*)$, we have that $h_r(x_*,u_*) = \beta h_{r-1}(x_*,u_*) = \dots = \beta^r h(x_*) >0$. The rest follows from continuity of the functions $h_i$ and $w$.}
    Since $B_{\bar{c}} \subseteq \bigcap_{i=0}^r \Sc_i$, there exist $\alpha^*_2>\alpha_1^*,$ $\gamma^*_{2} > \gamma^*_{1}$, such that with $\alpha > \alpha^*_2, \gamma > \gamma^*_{2}$, for each $(x,u)\in B_{\bar{c}}$, it holds that
    \begin{align*}
        &\frac{\partial h_r}{\partial x}(x,u)f(x,u) \!-\! \epsilon \frac{\partial h_r}{\partial u}(x,u) \frac{\partial w}{\partial u}(u)^\top \nabla \Phi(x) \! \geq \! -\gamma h_r(x,u), 
        \\
        &-\epsilon \nabla b(u)^\top \frac{\partial w}{\partial u}(u)^\top \nabla \Phi(x) \geq -\alpha b(u).
    \end{align*}
    Hence,    $g_{\epsilon,\alpha,\gamma}(x,u) = -\epsilon \frac{\partial w}{\partial u}(u)^\top \nabla \Phi(x)$ for all $(x,u) \in B_{\bar{c}}$, for $\alpha > \alpha^*_2, \gamma > \gamma^*_{2}$.
    Since $\frac{\partial w}{\partial u}(u_*)^\top \nabla\Phi(x_*) = 0$, then for all $(x,u) \in B_{\bar{c}}$,
    \begin{align*}
    &        \norm{g_{\epsilon,\alpha,\gamma}(x,u)}  
     \leq
        \epsilon \|\frac{\partial w}{\partial u}(u)^\top \nabla\Phi(x)-\frac{\partial w}{\partial u}(u_*)^\top \nabla\Phi(x_*)\|\\
& \leq        \epsilon \|\frac{\partial w}{\partial u}(u)^\top \nabla\Phi(x)-\frac{\partial w}{\partial u}(u)^\top \nabla\Phi(w(u))\|
\\
& \quad +      \epsilon \|\frac{\partial w}{\partial u}(u)^\top \nabla\Phi(w(u))-\frac{\partial w}{\partial u}(u_*)^\top \nabla\Phi(w(u))\| 
\\
& \quad +
        \epsilon \|\frac{\partial w}{\partial u}(u_*)^\top \nabla\Phi(w(u))-\frac{\partial w}{\partial u}(u_*)^\top \nabla\Phi(w(u_*))\|
        \\
&        \leq
        \epsilon L_1 
        \norm{x-w(u)} + {\epsilon (L_1 l_w + \tilde{l}_w M_{\Phi} )} \norm{u-u_*},
    \end{align*}
    where $l_w$ denotes the Lipschitz constant of $w$ in $\bigcap_{i=0}^r \Sc_i$,
    $\tilde{l}_w$ is the Lipschitz constant of $\frac{\partial w}{\partial u}(u)$ in $\bigcap_{i=0}^r \Sc_i$, and $M_{\Phi}$ is a bound on $\norm{\nabla\Phi(w(u))}$ in $\bigcap_{i=0}^r \Sc_i$. Note that such constants exist because $\bigcap_{i=0}^r \Sc_i$ is compact by assumption.

    \emph{$\bullet$ \underline{Showing the time derivative of 
    $\tilde{V}$ is negative definite} \underline{on $B_{\bar{c}}\setminus(x_*,u_*)$}:}
    Hence, 
    $\frac{d\tilde{V}}{dt}(x,u)\leq -\zeta^\top M_\epsilon\zeta$, for all $(x,u)\in B_{\bar{c}}$, where
    $  \zeta = [\norm{x-w(u)}, \norm{u-u_*}]$ and
    \begin{equation*}
        \begin{aligned}
            M_\epsilon&=
            \left(
            \begin{smallmatrix}
                d_3 - d_4 L_1 \epsilon & -{\frac{\epsilon L_1}{2}} - \frac{d_4 L_1 (l_w+1) \epsilon }{2} \\
                -{\frac{\epsilon L_1}{2}} - \frac{d_4 L_1 (l_w+1) \epsilon }{2} &\epsilon c
            \end{smallmatrix}
            \right).
        \end{aligned}
    \end{equation*}
    By taking $\alpha>\alpha_2^*$, $\gamma>\gamma_2^*$, $\epsilon < \epsilon^*_1 := \frac{d_3}{2 d_4 L_1}$, and 
    \begin{align*}
        \epsilon < \epsilon^*_2 := \frac{ \frac{d_3}{2} c }{ \Big( \frac{ L_1}{2} + \frac{d_4 L_1 (l_w+1) }{2} \Big)^2 },
    \end{align*}
    we get
    $\frac{d \tilde{V}}{dt}(x,u) < 0$ for all $(x,u) \in B_{\bar{c}} \backslash (x_*,u_*)$.  

    \emph{$\bullet$ \underline{Showing the time derivative of $\tilde{V}$ is negative definite} \underline{on $(\bigcap_{i=0}^r \Sc_i) \setminus B_{\bar{c}}$}:}
    Now let us show that there exists $\epsilon^*_3 > 0$, $\alpha^*_3>\alpha^*_2$ and $\gamma^*_3 > \gamma^*_2$ such that if $\epsilon < \epsilon_3^*$, $\alpha > \alpha_3^*$, $\gamma > \gamma_3^*$, then $\frac{d \tilde{V}}{dt}(x,u) < 0$ for all $(x,u)\in (\bigcap_{i=0}^r \Sc_i) \setminus B_{\bar{c}}$. To do so, we bound $\|g_{\epsilon,\alpha,\gamma}(x,u)\|$ by employing the results in Section \ref{subsec:auxiliary}. We start by bounding $\|g_\epsilon(x,u)\|$ (defined in~\eqref{eq:g_epsilon}):
    \begin{align}\label{eq:bound_on_g_epsilon}
        & \|g_\epsilon(x,u)\|\leq
        (1+\frac{1}{\sigma})\|f(x,u)\|+\epsilon(1+\frac{1}{\sigma})\|\frac{\partial w}{\partial u}(u)^\top \nabla\Phi(x)\| \notag
        \\
        & = (1+\frac{1}{\sigma})\|f(x,u)-f(w(u),u)\| \notag
        \\
        \notag
        & \quad + \epsilon(1+\frac{1}{\sigma})\|\frac{\partial w}{\partial u}(u)^\top \nabla\Phi(x)-\frac{\partial w}{\partial u}(u_*)^\top \nabla\Phi(x_*)\| \notag
        \\
        \notag
        & \leq (1+\frac{1}{\sigma})l_f\norm{x-w(u)} 
        + (1+\frac{1}{\sigma}) \epsilon L_1 \norm{x-w(u)} \\
        & + {\epsilon (1+\frac{1}{\sigma}) (L_1 l_w + \tilde{l}_w M_{\Phi} )} \norm{u-u_*} ,
    \end{align}
    where in the first inequality we used Lemma \ref{lem:ric_bound}. Since $g_{\epsilon,\alpha,\gamma}$ is continuous with respect to $\alpha$ and $\gamma$ (cf. Lemma~\ref{lem:continuity-alpha-gamma}, because $\bigcap_{i=0}^r\Sc_i$ is compact and the assumptions of Proposition~\ref{prop:feasibility} hold), by Lemma~\ref{lem:ric-bound-sgf} we have that, for any $\delta >0$, there exist $\alpha_\delta$ and $\gamma_\delta$, such that for $\alpha >\alpha_\delta$ and $\gamma>\gamma_\delta$,
     \begin{align}\label{eq:bound_g_epsilon_alpha_gamma}
        & \norm{g_{\epsilon,\alpha,\gamma}(x,u)} \leq \norm{g_{\epsilon}(x,u)} + \delta 
        \leq (1\!+\!\frac{1}{\sigma}) \cdot
        \\
        & \quad
        \cdot \Big( 
        (\epsilon L_1 \! + \! l_f)\norm{x \! - \! w(u)} \! + \! {\epsilon (L_1 l_w + \tilde{l}_w M_{\Phi})} \norm{u-u_*}
        \Big) \! + \! \delta, \notag
    \end{align}   
    By inserting \eqref{eq:bound_g_epsilon_alpha_gamma} into \eqref{eq:dvtilde_dt}, we obtain
    \begin{equation}\label{eq:dtildevdt_final_bound}
    \begin{aligned}
        \frac{d\tilde{V}}{dt}(x,u)&\leq -\zeta^\top M'_\epsilon\zeta + \|x-w(u)\|\delta \\&\leq -\lambda_{\text{min}}(M'_{\epsilon})\norm{\zeta}^2 + d_4 \norm{\zeta} \delta,
    \end{aligned}    
    \end{equation}
     where
    \begin{align*}
        M'_{\epsilon} &= \begin{pmatrix}
            M'_{\epsilon_{1,1}} & M'_{\epsilon_{1,2}} \\
            M'_{\epsilon_{2,1}} &  M'_{\epsilon_{2,2}}
        \end{pmatrix},\\
        M'_{\epsilon_{1,1}} &= d_3 - d_4(1+\frac{1}{\sigma})(l_f+\epsilon L_1), \ M'_{\epsilon_{2,2}} = \epsilon c, \\
        M'_{\epsilon_{1,2}} &=M'_{\epsilon_{2,1}}= -\frac{\epsilon L_1}{ 2 } - {
        \frac{\epsilon d_4 (1+\frac{1}{\sigma}) (L_1 l_w + \tilde{l}_w M_{\Phi}) }{2}
        }.
    \end{align*}
    Note that by taking 
    $\epsilon < \epsilon_3^* := \frac{d_3 - d_4(1+\frac{1}{\sigma})l_f}{2}$
    and
    \begin{align*}
        \epsilon < \epsilon_4^* := \frac{ c (d_3 - d_4(1+\frac{1}{\sigma}) l_f) }{ 2\Big( \frac{L_1}{2} + {
        \frac{ d_4 (1+\frac{1}{\sigma}) (L_1 l_w + \tilde{l}_w M_{\Phi}) }{2}
        } \Big)^2 },
    \end{align*}
    $M'_{\epsilon}$ is positive definite.
    It is important to note that $\epsilon_3^* > 0$
    and $\epsilon_4^* > 0$
    thanks to item~\ref{it:d_3-d_4} in the theorem hypotheses.
    Now, observe that $\delta$ can be chosen freely, determining the selection on $\alpha_\delta,\gamma_\delta$, by Lemma \ref{lem:ric-bound-sgf}. Select $\delta = \frac{\bar{c} \lambda_{\text{min}}(M'_{\epsilon}) }{2 d_4}$ and $\alpha_3^*>\max\{\alpha_2^*,\alpha_\delta\}$, $\gamma_3^*>\max\{\gamma_2^*,\gamma_\delta\}$. Equation \eqref{eq:dtildevdt_final_bound} becomes:
    \begin{equation*}
        \frac{d}{dt}\tilde{V}(x,u)\leq \lambda_{\text{min}}(M'_{\epsilon})\norm{\zeta}(-\|\zeta\| + \frac{\bar{c}}{2}) .
    \end{equation*}
    Note that, by definition of $B_{\bar{c}}$, we have $\|\zeta\|>\frac{\bar{c}}{2}$, for $(x,u)\in \bigcap_{i=0}^r \Sc_i \backslash B_{\bar{c}}$. Thus, for all $(x,u)\in \bigcap_{i=0}^r\Sc_i \backslash B_{\bar{c}}$, we have $\frac{d}{dt}\tilde{V}(x,u) < 0$.     Finally, the result follows by taking $0 < \epsilon < \min\{ \epsilon_1^*, \epsilon_2^*, \epsilon_3^*, \epsilon_4^* \}$ and $\alpha > \alpha_3^*$, $\gamma > \gamma_3^*$.
\end{proof}

\subsection{Auxiliary results for the proof of Theorem
\ref{thm:global_convergence}}\label{subsec:auxiliary}

\begin{lemma}\label{lem:alpha-proximity}
    Let $\Kc\subseteq\real^n$ be a compact set.
    Let $F:\Kc\to\real^n$ and $\tilde{F}:\Kc\times\real_{\geq0}\to\real^n$ be continuous maps with the property that for each $x\in\Kc$,
$       \lim\limits_{\alpha\to\infty}\tilde{F}(x,\alpha) = F(x)$.
    Then, for any $\delta > 0$, there exists $\alpha_{\delta}^* > 0$
    such that for $\alpha > \alpha_{\delta}^*$,
    \begin{align*}
        \sup\limits_{x\in\Kc} \| \tilde{F}(x,\alpha) - F(x) \| \leq \delta.
    \end{align*}
\end{lemma}
\begin{proof}
    By assumption, for each $x\in\Kc$, there exists $\alpha_{x,\delta} > 0$ such that for $\alpha > \alpha_{x,\delta}$, $\norm{ \tilde{F}(x,\alpha) - F(x) } < \delta$.
    Given $\alpha > \alpha_{x,\delta}$, since $x\mapsto\tilde{F}(x,\alpha)$ is continuous, this means that there exists a neighborhood $\Nc_{x}$ of $x$ such that $\norm{\tilde{F}(y,\alpha)-F(y)} < \delta$ for all $y\in\Nc_x$.
    Now, note that $\bigcup_{x\in\Kc} \Nc_x$ is an open cover of $\Kc$. Since $\Kc$ is compact, there exists a finite subcover $\bigcup_{i=1}^N \Nc_{x_i}$.
    The result follows by letting $\alpha_{\delta}^* = \max\limits_{i\in[N]} \alpha_{x_i,\delta}^*$.
\end{proof}

\begin{lemma}\longthmtitle{Continuity with respect to $\alpha$, $\gamma$}\label{lem:continuity-alpha-gamma}
    Suppose that the Assumptions of Proposition~\ref{prop:feasibility} hold. Define
    $\bar{g}(\epsilon,\alpha,\gamma,x,u)=g_{\epsilon,\alpha,\gamma}(x,u)$.
    Then, there exists $\breve{\alpha}, \breve{\gamma} > 0$ such that for any $\epsilon > 0$, $(x,u)\in\mathcal{S}_r$ and $\alpha > \breve{\alpha}$, $\gamma > \breve{\gamma}$, $\bar{g}$ is continuous at $(\epsilon, \alpha, \gamma, x, u)$.
\end{lemma}
\begin{proof}
    As shown in the proof of Proposition~\ref{prop:feasibility}, the assumptions therein guarantee that the QP defining $\bar{g}$ is strictly feasible for some $\alpha = \breve{\alpha}$ and $\gamma = \breve{\gamma}$, and hence Slater's condition holds.
    This also implies that Slater's condition holds for the QP defining $\bar{g}$ for any $\alpha > \breve{\alpha}$, $\gamma > \breve{\gamma}$. 
    By~\cite[Theorem 5.3]{AVF-JK:85}, 
    $\bar{g}$ is continuous at $(\epsilon,\alpha,\gamma,x,u)$ for all $\alpha > \breve{\alpha}$, $\gamma > \breve{\gamma}$.
\end{proof}

Other conditions under which continuity of $\bar{g}$ with respect to $\alpha$, $\gamma$ holds can be found in~\cite{PM-AA-JC:24-ejc}.
Next, define:
\begin{equation}\label{eq:g_epsilon}
\begin{aligned}
    &g_{\epsilon}(x,u) := \argmin_p \ \frac{1}{2}\|p+\epsilon \frac{\partial w}{\partial u}(u)^\top \nabla \Phi(x)\|^2,\\
    &\quad \mathrm{s.t.}: \ \delta(b(u))\nabla b(u)^\top p\geq0\\
    &\quad \quad \quad \ \delta(h_r(x,u))\Big(\frac{\partial h_r}{\partial u}(x,u)p+\frac{\partial h_r}{\partial x}(x,u)f(x,u)\Big)\geq 0,
\end{aligned}
\end{equation}
where the function $\delta:\real \to \real$ is $\delta(0)=1$ and $\delta(p)=0$ for $p\neq 0$. Further, let
\begin{align*}
    \Fc(x,u) := \begin{pmatrix}
        f(x,u) \\
        -\epsilon \frac{\partial w}{\partial u}(u)^\top \nabla \Phi(x)
    \end{pmatrix}.
\end{align*}

\begin{lemma}\label{lem:ric_bound}
    Let the assumptions of Proposition \ref{prop:feasibility} hold. Consider a compact set $\Cc_r\subseteq\Sc_r$. Assume that for all $(x,u)\in\Cc_r$, if $b(u) = 0$ and $h_r(x,u) = 0$, then $\nabla b(u)$ and $\frac{\partial h_r}{\partial u}(x,u)$ are linearly independent. Then, there exists a constant $\sigma$, such that $\|g_{\epsilon}(x,u)\| \leq (1+\frac{1}{\sigma}) \|\Fc(x,u)\|$, for all $(x,u)\in\Cc_r$.
\end{lemma}
\begin{proof}
    Notice that, when $b(u)>0$ and $h_r(x,u)>0$, then trivially $\|g_\epsilon(x,u)\|\leq \|\Fc(x,u)\|$. For the rest of the proof, we assume that both $b(u)=0$ and $ h_r(x,u)=0$, as the case where either $b(u)=0$ or $h_r(x,u)=0$ is a special case thereof.
    
    By Lemma \ref{lem:conditions-sigma-strong-mfcq}, there exists $\sigma\in(0,1)$, such that for any $(x,u)\in \Cc_r$, there exists $q_{(x,u)}\in\real^m$ such that \eqref{eq:sigma-strong-mfcq} holds; and thus $q_{(x,u)}$ is feasible for \eqref{eq:g_epsilon}. In what follows, denote:
    \begin{equation*}
        e  := (0,\ q_{(x,u)}+\epsilon \frac{\partial w}{\partial u}(u)^\top \nabla \Phi(x)) \in\real^{n+m}.
    \end{equation*}
    Notice that \eqref{eq:sigma-strong-mfcq} implies that:
    \begin{subequations}\label{eq:sigma-strong-simplified}
        \begin{align}
            \label{eq:hr>0}
            &\langle \Fc + e , \ \nabla h_r(x,u) \rangle \geq \sigma \norm{(f(x,u), \ q_{(x,u)})}\norm{\nabla h_r(x,u)}, \\ 
            \label{eq:b>0}
            &\langle \Fc + e , \ (0,\nabla b(u) \rangle \geq \sigma \norm{ q_{(x,u)}}\norm{\nabla b(x,u)}.
        \end{align}
    \end{subequations}
    In what follows, we omit the dependence of $\Fc$ and $e$ on $(x,u)$, for brevity.
    Consider any $(x,u)\in \Cc_r$. We first show that the statement holds trivially in the following two cases:
        \paragraph*{a)} If $p=-\epsilon \frac{\partial w}{\partial u}(u)^\top \nabla \Phi(x)$ is feasible for \eqref{eq:g_epsilon}, then $g_\epsilon(x,u) = -\epsilon \frac{\partial w}{\partial u}(u)^\top \nabla \Phi(x)$ and the statement holds.
        \paragraph*{b)} If $\|e \|\leq \|\Fc \|$, then
        \begin{align*}
            \|g_\epsilon(x,u)+\epsilon \frac{\partial w}{\partial u}(u)^\top \nabla \Phi(x)\|&\leq \|q_{(x,u)}+\epsilon \frac{\partial w}{\partial u}(u)^\top \nabla \Phi(x)\|\\
            &=\|e \|\leq \|\Fc \|,
        \end{align*}
        where the first inequality is because $q_{(x,u)}$ is feasible for \eqref{eq:g_epsilon}. Thus, $\|g_\epsilon(x,u)\|\leq \|\epsilon \frac{\partial w}{\partial u}(u)^\top \nabla \Phi(x)\|+\|\Fc \|\leq (1+\frac{1}{\sigma}) \|\Fc \|$.    
        
        In the rest of the proof, we consider the case where both
        \begin{enumerate}[{\it (i)}]
        \item $p=-\epsilon \frac{\partial w}{\partial u}(u)^\top \nabla \Phi(x)$  is not feasible for \eqref{eq:g_epsilon}, and
        \item $\|e \|> \|\Fc \|$.
        \end{enumerate}
        
        Let us define the linear functions $r_1,r_2:\real\to\real$ by
        \begin{align*}
            &r_1(t) :=\langle\Fc +te  , \ \nabla h_r(x,u)\rangle = t\frac{\partial h_r}{\partial u}(x,u)q_{(x,u)} + \\
            &\quad \quad (t-1) \frac{\partial h_r}{\partial u}(x,u)\epsilon \frac{\partial w}{\partial u}(u)^\top \nabla \Phi(x) + \frac{\partial h_r}{\partial x}(x,u)f(x,u), \\
            &r_2(t):=t\nabla b(u)^\top q_{(x,u)} +(t-1) \nabla b(u)^\top\epsilon \frac{\partial w}{\partial u}(u)^\top \nabla \Phi(x).
        \end{align*}
        Let us show that there exist $t_* \in (0,1)$, $j_*\in \{1,2\}$ such that
        \begin{subequations}
        \label{eq:boundary_point}
        \begin{align}
            \label{eq:t_star}
            &r_{j_\star}(t_*) = 0, \\
            \label{eq:t_TS}
            \begin{split}
                &\mbox{For} \ t \leq 1, \ r_1(t),r_2(t)\geq 0 \iff r_{j_\star}(t) \geq 0.
            \end{split}
        \end{align}
        \end{subequations}
        By item \textit{(i)}, there exists a nonempty set $J\subseteq \{1,2\}$ with
        \begin{subequations}
        \begin{align}
            \label{eq:J}
            r_j(0) < 0,& \mbox{ for all } j \in J, \\
            \label{eq:I_minus_J}
            r_i(0) \geq 0,& \mbox{ for all } i \in \{1,2\} \setminus J.
        \end{align}
    \end{subequations}
    Moreover, due to \eqref{eq:sigma-strong-simplified},
    $r_1(1), r_2(1) > 0$.\footnote{The RHS of \eqref{eq:sigma-strong-simplified} is strictly positive, since: a) item \textit{(ii)} implies that $(f(x,u),q_{(x,u)})=\Fc+e\neq 0$ and $q_{(x,u)}\neq 0$, and b) $\nabla b(u)$ and $\nabla h_r(x,u)$ are non-zero, from linear independence of $\nabla b(u)$ and $\frac{\partial h_r}{\partial u}(x,u)$.}
    Due to linearity, this together with \eqref{eq:J} implies that, for each $j \in J$, $r_j$ is strictly increasing, therefore there exists a unique scalar $t_{j} \in (0,1)$ such that $r_j(t_{j}) = 0$. Picking\footnote{When the solution of the maximization is non-unique, the proof follows identically by selecting any of the maximizers and the corresponding index~$j$.} $t_* = \max_{j\in J}\{t_{j}\}$ and $j_* = \argmax_{j\in J}\{t_{j}\}$, we have \eqref{eq:t_star}. Further, again due to strict increase,
    \begin{align}
        \label{eq:aux_rj_geq_0}
        t \in [t_*, 1] \implies r_j(t) \geq 0, \mbox{ for all } j \in J.
    \end{align}
    What is more,~\eqref{eq:I_minus_J} together with $r_1(1), r_2(1) > 0$ yields
    \begin{align}
        \label{eq:monotonicity_consequence}
        t \in [0,1] \implies r_i(t) \geq 0, \mbox{ for all } i \in \{1,2\} \setminus J.
    \end{align}
    Combining \eqref{eq:t_star}, \eqref{eq:aux_rj_geq_0} and \eqref{eq:monotonicity_consequence}, we obtain \eqref{eq:t_TS}.
    Now, consider $\tilde{t} := \frac{\langle\Fc , \ e  \rangle}{\|e \|^2}$, and notice that $\tilde{t}\in (-1,1)$ due to item \textit{(ii)}. Let us also define
    $\tilde{e}:= (1+\tilde{t})e$, $\tilde{\Fc} := \Fc  - \tilde{t}e$.
    Note that $\tilde t e $ is the projection of $\Fc $ on the linear subspace generated by $e$.
    Thus, $\tilde \Fc  \perp \tilde e $, $\|\tilde t e \| \leq \|\Fc \|$, $\|\tilde \Fc \| \leq \|\Fc \|$ (these are properties of the orthogonal projection) and $\tilde \Fc  + \tilde e  = \Fc  + e $.
    We consider two cases.

    \textbf{Case A:} $p_*:= -(1+\tilde t)\epsilon \frac{\partial w}{\partial u}(u)^\top \nabla \Phi(x) -\tilde{t} q_{(x,u)}$ is feasible for \eqref{eq:g_epsilon}. Then:
    \begin{align*}
            \|g_\epsilon(x,u)+\epsilon \frac{\partial w}{\partial u}(u)^\top \nabla \Phi(x)\|&\leq \|p_*+\epsilon \frac{\partial w}{\partial u}(u)^\top \nabla \Phi(x)\|\\
            &=\tilde t\|e \|\leq \|\Fc \|.
        \end{align*}
    Thus, $\|g_\epsilon(x,u)\|\leq \|\epsilon \frac{\partial w}{\partial u}(u)^\top \nabla \Phi(x)\|+\|\Fc \|\leq (1+\frac{1}{\sigma}) \|\Fc \|$.
    
    \textbf{Case B:} $p_*= -(1+\tilde t)\epsilon \frac{\partial w}{\partial u}(u)^\top \nabla \Phi(x) -\tilde{t} q_{(x,u)}$ is not feasible for \eqref{eq:g_epsilon}. This implies that $r_1(-\tilde t)<0$ or $r_2(-\tilde t)<0$. From \eqref{eq:t_TS}, this implies that $r_{j_*}(-\tilde t)<0$. We consider the two following sub-cases:
    
    \textit{Case B1:} $j_*=1$. Then we have $\langle\tilde \Fc , \nabla h_r(x,u) \rangle = r_1(-\tilde t) < 0$. Combining this with \eqref{eq:hr>0} and the fact that $\tilde \Fc  + \tilde e  = \Fc  + e $, we get
        \begin{align*}
            &\langle\tilde e , \nabla h_r(x,u) \rangle \! > \! \langle\tilde \Fc + \tilde e, \nabla h_r(x,u) \rangle = \langle \Fc + e , \nabla h_r(x,u) \rangle\\
            &\geq \sigma \norm{(f(x,u), \ q_{(x,u)})}\norm{\nabla h_r(x,u)} \! = \! \sigma \norm{\Fc  \! + \! e }\norm{\nabla h_r(x,u)}\\
            & = \sigma \norm{\tilde\Fc  + \tilde e }\norm{\nabla h_r(x,u)} \geq \sigma \norm{ \tilde e }\norm{\nabla h_r(x,u)}, 
        \end{align*}
        where the last inequality comes from the fact that $\tilde \Fc  \perp \tilde e $. Multiplying both sides by $\frac{t_*}{1+\tilde t}>0$, we obtain
        \begin{align*}
        &\sigma \| t_* e \| \|\nabla h_{r}(x,u) \|
        = \sigma\Big\| \frac{t_*}{1 + \tilde t} \tilde e \Big\| \left\| \nabla h_{r}(x,u) \right\| \\
        &\leq \dotprod{\frac{t_*}{1 + \tilde t} \tilde e}{\nabla h_{r}(x,u)} = \dotprod{t_* e}{\nabla h_{r}(x,u)} = \\
        &r_1(t_*) \! - \! \dotprod{\Fc}{\nabla h_{r}(x,u)} \! = \! - \! \dotprod{\Fc}{\nabla h_{r}(x,u)} \! \leq \! \|\Fc\| \|\nabla h_{r}(x,u)\|.
    \end{align*}
    By~\eqref{eq:boundary_point}, we have $r_1(t_*) =0 $ and $r_2(t_*)\geq 0$. Hence, $q_* := (t_*-1)\epsilon \frac{\partial w}{\partial u}(u)^\top \nabla \Phi(x) + t_*q_{(x,u)}$ is feasible for \eqref{eq:g_epsilon}. Thus:
    \begin{align*}
            \|g_\epsilon(x,u)+\epsilon \frac{\partial w}{\partial u}(u)^\top \nabla \Phi(x)\|&\leq \|q_*+\epsilon \frac{\partial w}{\partial u}(u)^\top \nabla \Phi(x)\|\\
            &=t_*\|e\|\leq \frac{1}{\sigma}\|\Fc \|.
        \end{align*}
    and the bound follows.
    
    \textit{Case B2:} $j_*=2$. We have $\langle\tilde \Fc , (0,\nabla b(u)) \rangle = r_2(-\tilde t) < 0$. Combining this with \eqref{eq:b>0} and $\tilde \Fc  + \tilde e  = \Fc  + e $, we get
    \begin{align*}
            &\langle\tilde e , (0,\nabla b(u)) \rangle > \langle\tilde \Fc + \tilde e , \nabla (0,\nabla b(u) \rangle = \langle \Fc + e , (0,\nabla b(u)) \rangle 
            \\
            &\geq \! \sigma \norm{ q_{(x,u)}} \! \norm{\nabla b(u)} \! = \! \sigma \Big\| (0,-\epsilon \frac{\partial w} {\partial u}(u)^\top \nabla \Phi(x)) \! + \! e \Big\| \! \norm{\nabla b(u)}\\
            & = \sigma \norm{\Fc -(f(x,u),0)  + e }\norm{\nabla b(u)}\\
            & = \sigma \norm{\tilde\Fc  + \tilde e -(f(x,u),0)}\norm{\nabla b(u)} \geq \sigma \norm{ \tilde e }\norm{\nabla b(u)}, 
        \end{align*}
        where the last inequality comes from the fact that $\tilde \Fc \perp \tilde e $ and $(f(x,u),0) \perp \tilde e$. The rest follows as in the previous case.
\end{proof}

\begin{lemma}\longthmtitle{Upper bound on SGF}\label{lem:ric-bound-sgf}
    Let the assumptions of Proposition \ref{prop:feasibility} hold. Consider a compact set $\Cc_r\subseteq\Sc_r$. Assume that for all $(x,u)\in\Cc_r$, if $b(u) = 0$ and $h_r(x,u) = 0$, then $\nabla b(u)$ and $\frac{\partial h_r}{\partial u}(x,u)$ are linearly independent. Then, there exists $\sigma >0$ such that: for any $\delta > 0$, there exist $\alpha_{\delta} > 0$ and $\gamma_{\delta} > 0$
    such that by taking $\alpha > \alpha_{\delta}$ and $\gamma > \gamma_{\delta}$, we have
    \begin{align*}
        \norm{g_{\epsilon,\alpha,\gamma}(x,u)} \leq (1+\frac{1}{\sigma})\norm{\Fc } + \delta.
    \end{align*}
\end{lemma}
\begin{proof}[\textbf{Proof of Lemma \ref{lem:ric-bound-sgf}}]
    By an argument analogous to~\cite[Proposition 4.4]{allibhoy2023control}, 
    and taking $\gamma = \alpha$, we have that for each $(x,u)\in\mathcal{C}_r$, $\lim\limits_{\alpha\to\infty} g_{\epsilon,\alpha,\gamma}(x,u) = g_{\epsilon}(x,u)$.
    Moreover, $g_{\epsilon,\alpha,\gamma}$ is continuous with respect to $\alpha$ and $\gamma$ for $\alpha > \breve{\alpha}$ and $\gamma > \breve{\gamma}$ by Lemma \ref{lem:continuity-alpha-gamma}.
    By Lemma~\ref{lem:alpha-proximity}, there exist $\alpha_{\delta},\gamma_{\delta} > 0$
    such that by taking $\alpha > \alpha_{\delta}$ and $\gamma > \gamma_{\delta}$, we have $\norm{g_{\epsilon,\alpha,\gamma}(x,u)} \leq \norm{g_{\epsilon}(x,u)} + \delta$. The result follows from Lemma~\ref{lem:ric_bound}.
\end{proof}

\bibliography{mybib.bib}
\bibliographystyle{IEEEtran}

\vspace*{-4ex}

\begin{IEEEbiography}[{\includegraphics[width=1in,height=1.2in,clip,keepaspectratio]{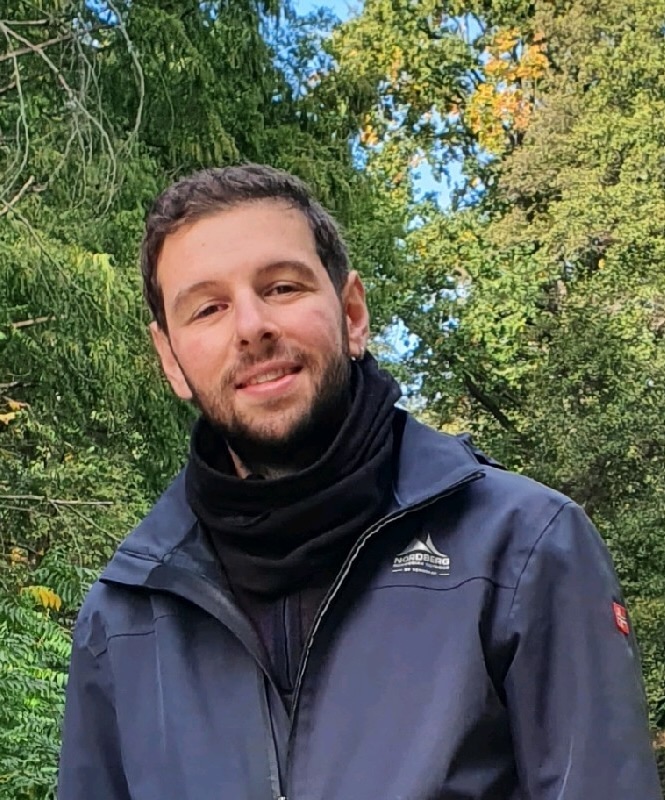}}]{Giannis Delimpaltadakis} (S'18, M'22) received his diploma in Electrical and Computer Engineering in 2017 from the National Technical University of Athens. He obtained his PhD with Cum Laude distinction in 2022 from Delft University of Technology. Currently, he is a senior research scientist at AI4I, Italy. His research interests include control theory (particularly stochastic and hybrid systems) and its intersections with formal methods, optimization and information theory.
\end{IEEEbiography}

\vspace*{-4ex}

\begin{IEEEbiography}[{\includegraphics[width=1in,height=1.2in,clip,keepaspectratio]{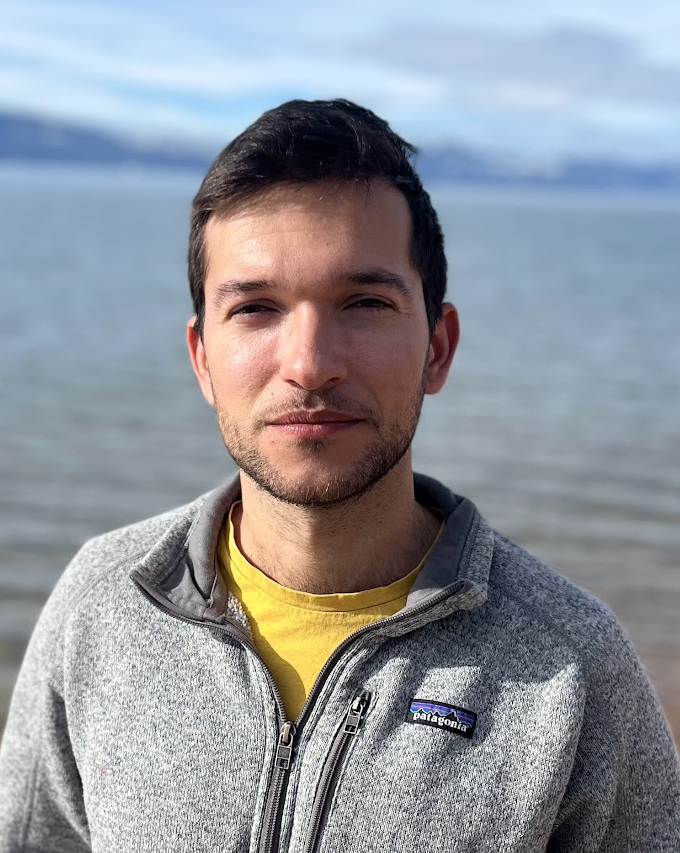}}]{Pol Mestres} received the Bachelor's degree in mathematics and the Bachelor's degreee in engineering physics from the Universitat Politècnica de Catalunya, Barcelona, Spain, in 2020, and the Master's and Ph.D degrees in mechanical engineering in 2021 and 2025 respectively from the University of California, San Diego, La Jolla, CA, USA. 
He is currently a postdoctoral scholar at the California Institute of Technology.
His research interests include safety-critical control, motion planning, and reinforcement learning.
\end{IEEEbiography}

\vspace*{-4ex}

\begin{IEEEbiography}[{\includegraphics[width=1in,height=1.2in,clip,keepaspectratio]{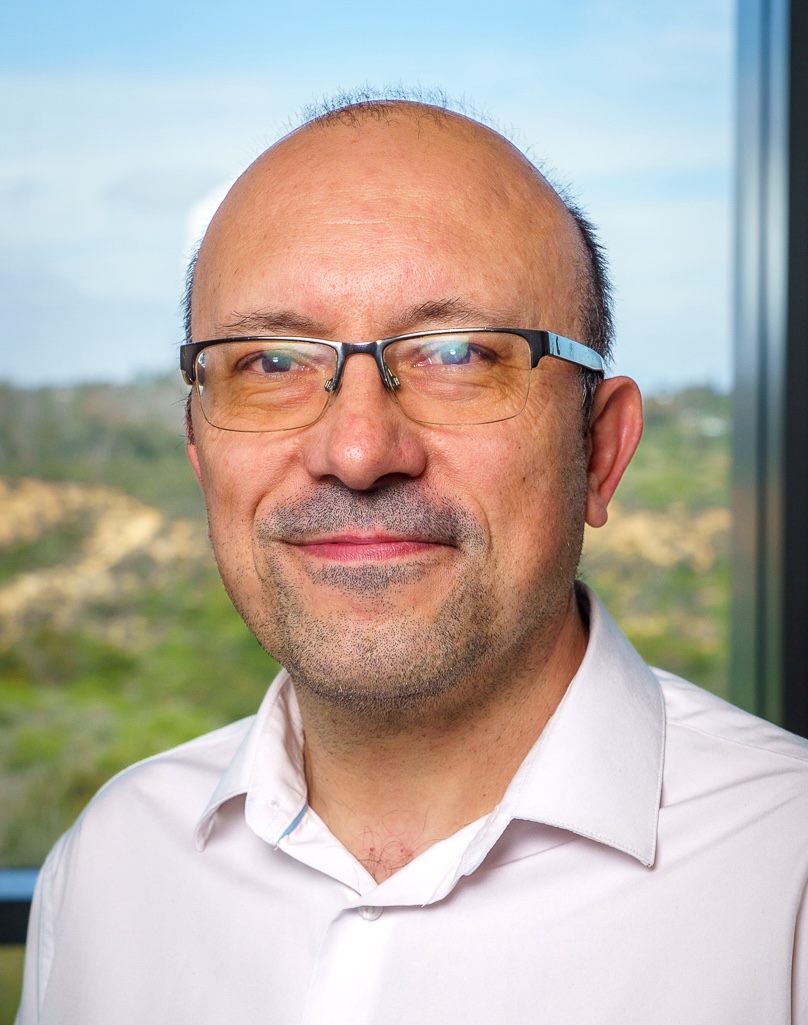}}]{Jorge
     Cort\'{e}s}(M'02, SM'06, F'14) received the Licenciatura degree in
   mathematics from Universidad de Zaragoza, Spain, in 1997, and the
   Ph.D. degree in engineering mathematics from Universidad Carlos III
  de Madrid, Spain, in 2001. He held postdoctoral positions with the
  University of Twente, Twente, The Netherlands, and the University of
  Illinois at Urbana-Champaign, Illinois, USA. 
  He is a Professor and Cymer Corporation Endowed Chair in High
  Performance Dynamic Systems Modeling and Control at the Department
  of Mechanical and Aerospace Engineering, UC San Diego, California,
  USA.  He is a Fellow of IEEE, SIAM, and IFAC.  His research
  interests include distributed control and optimization, network
  science, autonomy, learning, nonsmooth analysis, decision making under
  uncertainty, network neuroscience, and multi-agent coordination in
  robotic, power, and transportation networks.
\end{IEEEbiography}

\vspace*{-4ex}

\begin{IEEEbiography}[{\includegraphics[width=1in,height=1.2in,clip,keepaspectratio]{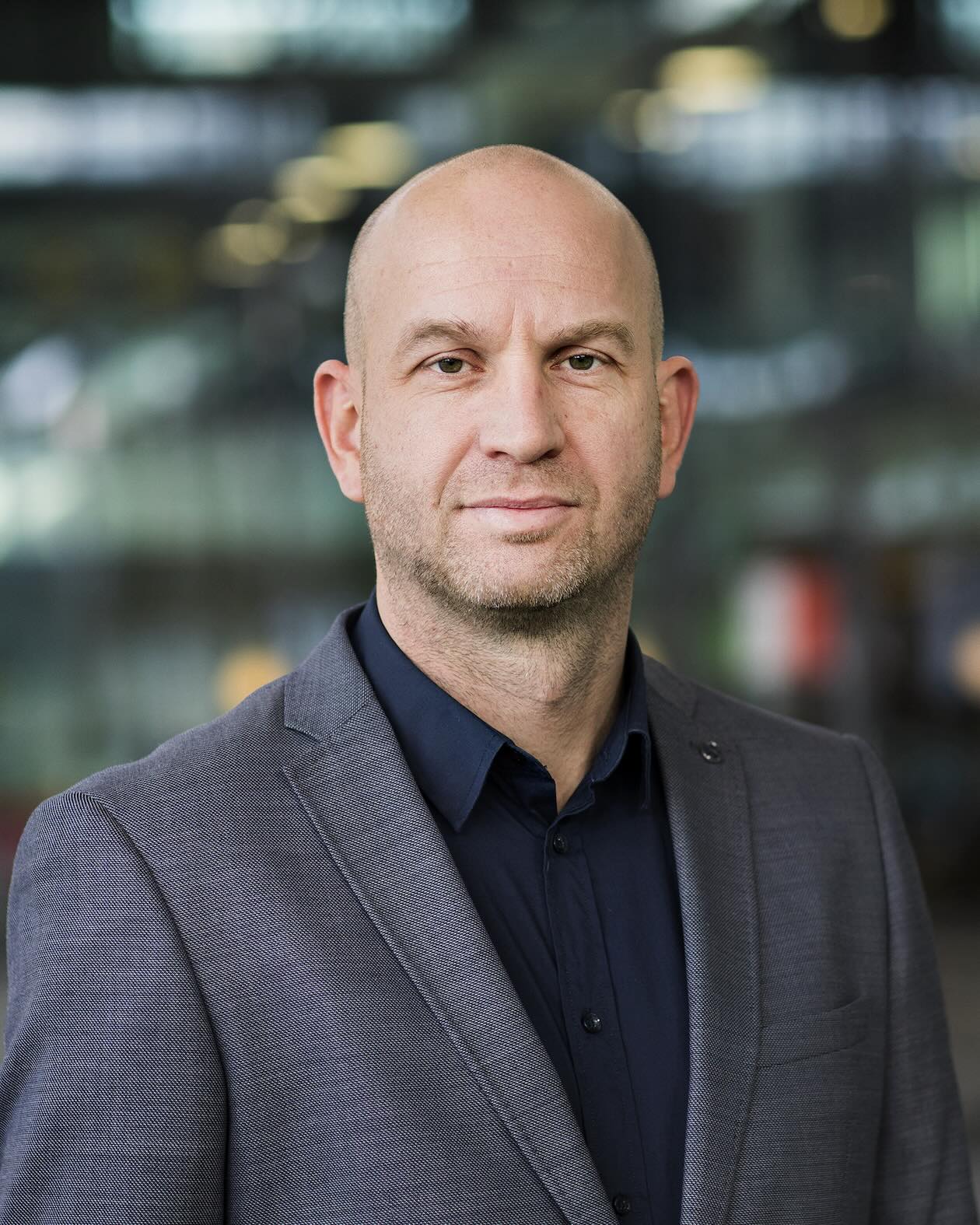}}]{Maurice Heemels}  received M.Sc. (mathematics) and Ph.D. (EE, control theory) degrees (summa cum laude) from the Eindhoven University of Technology (TU/e) in 1995 and 1999, respectively. From 2000 to 2004, he was with the Electrical Engineering Department, TU/e, as an assistant professor, and from 2004 to 2006 with TNO-Embedded Systems Institute as a Research Fellow. Since 2006, he has been with the Department of Mechanical Engineering, TU/e, where he is currently a Full Professor and Vice-Dean. He was a visiting professor at ETH, Switzerland, UCSB, USA, and University of Lorraine, France. He  is a Fellow of IEEE and IFAC. He was the recipient of the 2019 IEEE L-CSS Outstanding Paper Award and the Automatica Paper Prize 2020-2022. His current research includes hybrid and cyber-physical systems, networked, neuromorphic and event-triggered control, and model predictive control and their applications. \end{IEEEbiography}

}

\end{document}